\documentclass{amsart}
\usepackage{fdsymbol} 
\usepackage{tikz-cd} 

    \usepackage{etoolbox}
    \patchcmd{\section}{\scshape}{\large\bfseries}{}{}
    \makeatletter
    \renewcommand{\@secnumfont}{\bfseries}
    \makeatother
    
\usepackage{biblatex}
\usepackage{booktabs}
\bibliography{references}

\numberwithin{equation}{section}

\newtheorem{theorem}{Theorem}[section]
\newtheorem{corollary}[theorem]{Corollary}
\newtheorem{lemma}[theorem]{Lemma}
\newtheorem{proposition}[theorem]{Proposition}
\theoremstyle{definition}
\newtheorem{definition}[theorem]{Definition}
\newtheorem{remark}[theorem]{Remark}

\def\epi{\twoheadrightarrow}
\def\mono{\rightarrowtail}
\def\CC{\mathcal{C}}
\def\FF{\mathcal{F}}
\def\GG{\mathcal{G}}
\def\UU{\mathcal{U}}
\def\DD{\mathcal{D}}

\def\BB{\mathcal{B}}
\def\AA{\mathcal{A}}

\def\MM{\mathcal{M}}
\def\gg{\mathfrak{g}}
\def\ff{\mathfrak{f}}
\def\rr{\mathfrak{r}}
\def\ll{\mathfrak{l}}
\def\kk{\Bbbk}
\def\aa{\mathfrak{a}}
\def\bb{\mathfrak{b}}

\def\ZZ{\mathbb{Z}}

\title[On homology of Lie algebras]{On homology of Lie algebras \\ over commutative rings}

\author[Ivanov]{Sergei O. Ivanov} 
\address{
Laboratory of Modern Algebra and Applications,  St. Petersburg State University, 14th Line, 29b,
Saint Petersburg, 199178 Russia}
\email{ivanov.s.o.1986@gmail.com}

\author[Pavutnitskiy]{Fedor Pavutnitskiy}
\address{National Research University Higher School of Economics, Russian Federation}
\email{fedor.pavutnitskiy@gmail.com}

\author[Romanovskii]{Vladislav Romanovskii}
\address{
Laboratory of Modern Algebra and Applications,  St. Petersburg State University, 14th Line, 29b,
Saint Petersburg, 199178 Russia}
\email{Romanovskiy.vladislav.00@mail.ru}

\author[Zaikovskii]{Anatolii Zaikovskii}
\address{
Laboratory of Modern Algebra and Applications,  St. Petersburg State University, 14th Line, 29b,
Saint Petersburg, 199178 Russia}\email{anat097@mail.ru}

\thanks{The work is supported by: (1) Ministry of Science and Higher Education of the Russian Federation, agreement  075-15-2019-1619; (2) the grant of the Government of the Russian Federation for the state support of scientific research carried
out under the supervision of leading scientists, agreement 14.W03.31.0030 dated
15.02.2018; (3) the Russian Academic Excellence Project 5-100 within the framework of the Basic Research Program at HSE University; (4)  RFBR according to the research project 20-01-00030; (5) Russian Federation President Grant for Support of Young Scientists MK-681.2020.1.}

\begin{document}

\maketitle

\begin{abstract}
We study five different types of the homology of a Lie algebra over a commutative ring which are naturally isomorphic over fields. We show that they are not isomorphic over commutative rings, even over $\mathbb Z,$ and study connections between them. In particular, we show that they are naturally isomorphic in the case of a Lie algebra which is flat as a module.
As an auxiliary result we prove that the Koszul complex of a module $M$ over a principal ideal domain that connects the exterior and the symmetric powers  $0\to \Lambda^n M\to M \otimes \Lambda^{n-1} M \to \dots \to  S^{n-1}M \otimes M \to S^nM\to 0 $ is purely acyclic. 
\end{abstract}

\section*{Introduction} 

There are several equivalent purely algebraic definitions of group homology of a group $G.$ Namely one can define it using Tor functor over the group ring; or the relative Tor functor; or in the simplicial manner, as a simplicial derived functor of the functor of abelianization
(see \cite[Ch. II, \S 5]{quillen2006homotopical}). For Lie algebras over a field we can also define homology in several equivalent ways, including the homology of the  Chevalley–Eilenberg complex. Moreover, for a Lie algebra $\gg$ over a commutative ring $\kk,$ which is free as a $\kk$-module we also have several equivalent definitions (see \cite[Ch. XIII]{cartan1999homological}). However, in general, even in the case $\kk=\ZZ,$ these definitions are not equivalent. For example, if we consider $\gg=\ZZ/2$ as an abelian Lie algebra over $\ZZ,$ the higher homology of the Chevalley–Eilenberg complex vanishes but ${\sf Tor}_{2n+1}^{U\gg}(\ZZ,\ZZ)=\ZZ/2$ for any $n\geq 0.$  This paper is devoted to study of different types of homology of Lie algebras over commutative rings that are equivalent for Lie algebras over fields. For other approaches to similar questions see  \cite{dixmier1957homologie}, \cite{pirashvili2003algebra}, \cite{shukla1961cohomologie}, \cite{quillen1970co}.

Let $\gg$ be a Lie algebra over a commutative ring $\kk$ and $U\gg$ be its universal enveloping algebra. We consider the following variants of homology of $\gg$.
\begin{enumerate}
\item {\it Tor homology} is defined via the Tor functor: $$H^{\sf T}_*(\gg)={\sf Tor}^{U\gg}_*(\kk,\kk).$$
\item {\it Relative Tor homology} is defined via the relative Tor functor: $$H^{\sf RT}_*(\gg)={\sf Tor}^{(U\gg,\kk)}_*(\kk,\kk).$$
\item {\it Chevalley-Eilenberg homology} is defined as
$$H^{\sf CE}_*(\gg)=H_*({\sf CE}_\bullet(\gg)),$$
where ${\sf CE}_\bullet(\gg)$ is  the Chevalley-Eilenberg complex of $\gg.$ The components of ${\sf CE}_\bullet(\gg)$ are exterior powers  ${\sf CE}_n(\gg)=\Lambda^n \gg$ and the differential $\Lambda^n \gg\to \Lambda^{n-1} \gg$ sends $x_1\wedge \dots \wedge x_n$ to 
$$\sum_{i<j}(-1)^{i+j}   [x_i,x_j] \wedge x_1\wedge \dots \wedge \hat x_i \wedge \dots \wedge \hat x_j \wedge \dots \wedge x_n.$$

\item {\it Simplicial homology} is defined as comonad derived functors (see \cite{barr1969homology}) of the functor of abelianization ${\sf ab}:{\sf Lie} \to {\sf Mod}$ with respect to the comonad $\GG^{\sf S}$ of the free forgetful-adjunction of the category of Lie algebras with the category of sets
${\sf Sets} \rightleftarrows {\sf Lie}.$ 
$$H_{*+1}^{\sf S}(\gg)=L_*^{\GG^{\sf S}} {\sf ab}(\gg).$$
This derived functor can be equivalenly defined as a derived functor in the sense of Tierney-Vogel \cite{tierney1969simplicial} or in the sense of Quillen \cite[Ch. XIII]{cartan1999homological} but here we prefer the language of comonad derived functors. 

\item {\it Relative simplicial homology} is defined similarly to the simplicial homology but instead of the comonad $\GG^{\sf S}$ we consider the comonad $\GG^{\sf M}$  of another adjunction, of the adjunction of the category of Lie algebras with the category of $\kk$-modules
${\sf Mod} \rightleftarrows {\sf Lie}$
$$H^{\sf RS}_{*+1}(\gg)=L_*^{\GG^{\sf M}} {\sf ab}(\gg).$$
\end{enumerate}

We prove that under certain conditions some of the definitions of homology are equivalent. Namely, we prove the following two theorems.

\bigskip 

\noindent {\bf Theorem} (Theorem \ref{th_flat}). {\it If $\kk$ is any commutative ring and $\gg$ is flat as a module over $\kk,$ then all these types of homology are naturally isomorphic
$$H^{\sf S}_*(\gg) \cong H^{\sf T}_*(\gg)\cong H^{\sf CE}_*(\gg)\cong  H^{\sf RT}_*(\gg)\cong H^{\sf RS}_*(\gg).$$}

Note that a cohomological version of the isomorphism $H^{\sf S}_*(\gg)\cong H^{\sf T}_*(\gg)$ for a flat Lie algebra $\gg$ was proven  by Quillen \cite[Prop. 3.7]{quillen1970co}. 

\ 
  
\noindent {\bf Theorem} (Theorem \ref{th_RS}, Theorem \ref{th_hopf}). {\it  If $\kk$ is a principal ideal domain and $\gg$ is any Lie algebra over $\kk$, then there are natural isomorphisms $$H^{\sf CE}_*(\gg)\cong H^{\sf RT}_*(\gg)\cong H^{\sf RS}_*(\gg)$$ 
and natural isomorphisms for all types of the second homology
$$H^{\sf S}_2(\gg) \cong H^{\sf T}_2(\gg)\cong H^{\sf CE}_2(\gg)\cong  H^{\sf RT}_2(\gg)\cong H^{\sf RS}_2(\gg).$$
}

\noindent However, we provide an example of a Lie algebra $\gg$ over $\ZZ$ such that 
$$H_3^{\sf S}(\gg)\cong (\ZZ/2)^{\oplus \infty} , \hspace{1cm}  H_3^{\sf T}(\gg) \cong (\ZZ/2)^4,$$
$$  H_3^{\sf CE}(\gg)=H_3^{\sf RT}(\gg)=H^{\sf RS}_3(\gg)=0$$
(Proposition \ref{proposition_main_example}). We also construct an example of a commutative ring $\kk$ and a Lie algebra $\gg$ over $\kk$ such that 
$$H_2^{\sf CE}(\gg)\not\cong H_2^{\sf RT}(\gg)$$
(Proposition  \ref{prop_counter_H_2}) and another example such that 
$$H_3^{\sf CE}(\gg)\not \cong H_3^{\sf RS}(\gg)$$
(Proposition \ref{prop_counter_RSCE}). These examples show that the only possibility for isomorphism among  these five types of homology is an isomorphism between $H_*^{\sf RT}(\gg)$ and $H_*^{\sf RS}(\gg).$ However, we do not believe that they are isomorphic and we leave it here in the form of conjecture.

\ 

\noindent {\bf Conjecture.} {\it There exists a commutative ring $\kk$ such that the functors   $H_*^{\sf RT}$ and $H_*^{\sf RS}$ are non-isomorphic.}

\ 

The most complicated part in these examples was to compute the third simplicial homology $H_3^{\sf S}$ for some Lie algebras over $\ZZ.$ In order to do this we constructed a spectral sequence that we call Chevalley-Eilenberg spectral sequence. We consider derived versions of the  Chevalley-Eilenberg complex 
$$L_m{\sf CE}_\bullet(\gg): \hspace{1cm}  \dots\longrightarrow  L_m \Lambda^n (\gg) \longrightarrow  L_m \Lambda^{n-1}(\gg) \longrightarrow \dots  ,$$
where $L_m\Lambda^n(-)=L_m\Lambda^n(-,0)$ is the non-shifted Dold-Puppe  derived functor (see \cite{dold1961homologie}) of the functor of the exterior power. Homology of $L_m{\sf CE}_\bullet(\gg)$ is denoted by $H_{n,m}^{\sf CE}(\gg).$ The modules $H_{n,m}^{\sf CE}(\gg)$ are called {\it derived  Chevalley-Eilenberg homology}. 

\ 

\noindent {\bf Theorem} (Theorem \ref{th_CE}). 
{\it For any Lie algebra $\gg$ over any commutative ring $\kk$ there is a spectral sequence converging from derived Chevalley-Eilenberg homology to simplicial homology
$$H_{n,m}^{\sf CE}(\gg) \Rightarrow H_{n+m}^{\sf S}(\gg).$$}

Using this spectral sequence we show that there is an exact sequence
$$H_4^{\sf CE}(\gg)\to H_{2,1}^{\sf CE}(\gg)\to H_3^{\sf S}(\gg)\to H_{3}^{\sf CE}(\gg)\to 0.$$ 
This exact sequence helps to reduce  computations of $H_3^{\sf S}$ to computations of $H_{2,1}^{\sf CE}.$ It is still quite complicated functor but we found its description on the language of derived functors in the sense of Dold-Puppe developed by Breen and Jean \cite{breen1999functorial}, \cite{jeanphdthesis} (Theorem \ref{th_H21}). 
Namely, in Breen's notations we prove that there exists an exact sequence \[
\Omega \gg \otimes \gg \overset{\varphi}\longrightarrow \Omega \gg \longrightarrow H_{2,1}^{\sf CE}(\gg) \longrightarrow 0,    
\]
where $\Omega=L^1 \Lambda^2$ and
$
\varphi( \omega^n_2(a)\otimes b )=\omega_1^n([a,b])*\omega_1^n(a).
$

An important tool in our work is the Koszul complexes of a $\kk$-module which can be considered as a relation between functors of symmetric powers and exterior powers:
$$ 0 \to \Lambda^2 M \to M\otimes M \to S^2 M \to 0,$$
$$0 \to \Lambda^3 M  \to M \otimes \Lambda^2 M \to S^2M \otimes M \to S^3 M \to 0,$$
$$\dots  $$
\begin{equation}\label{eq_kos_intro}
0 \to \Lambda^n M \to M\otimes \Lambda^{n-1}M \to \dots \to S^{i}M\otimes \Lambda^{n-i} M \to \dots \to S^nM \to 0.    
\end{equation}

These complexes come from the general construction of a Koszul complex in commutative algebra (see Section \ref{section_Koszul} for details). 
It is known that, if $M$ is a flat $\kk$-module, then these complexes are acyclic \cite{illusie2006complexe}. 
We prove the following.

\ 

{\bf Theorem} (Theorem \ref{th_kos}).  {\it Let $M$ be a module over a commutative ring $\kk$ and $n\geq 1$ be a natural number. Assume that one of the following properties hold: 
\begin{enumerate}
    \item $\kk$ is a principal ideal domain;
    \item $M$ is flat;
    \item  $n \cdot 1_\kk$ is invertible in $\kk.$
\end{enumerate}
Then the Koszul complex \eqref{eq_kos_intro} is pure acyclic (see \cite{emmanouil2016pure}). Moreover, in the case (3) it is contractible.}

\

We also provide an example of a commutative ring $\kk$ and a module $M$ such that these complexes are not acyclic. Moreover, we provide an example of $\kk$ and $M$ such that the map
$$\Lambda^2 M \to M\otimes M, \hspace{1cm} a\wedge b \mapsto a\otimes b - b \otimes a$$
is not injective (Corollary \ref{cor_counter_kos}). 

In Appendix we show that, if for a $\kk$-module $M$ the map $\Lambda^2 M \to M\otimes M$ is not injective, then a Lie algebra given by $\gg=M\oplus  \Lambda^2 M,$ where $[a,b]=a\wedge b$ for $a,b\in M$ and $[\alpha,-]=0=[-,\alpha]$ for $\alpha\in \Lambda^2 M,$ does not satisfy the Poincar\'e–Birkhoff–Witt property. These examples are related to examples of Lie algebras constructed by Cohn  \cite{cohn1963remark} and Shirshov \cite{shirshov1953representation}.

\section{Reminder on the bar construction} 

\subsection{The bar construction for comonads}

Let $\Delta_+$ denote the category of finite ordinals, including the empty set. This category is called {\it augmented simplex category}. Its objects are natural numbers $n=\{0,\dots,n-1\}$ including $0=\emptyset$ and morphisms are order-preserving functions. Its objects are also denoted by $$[n]=n+1$$ for $n\geq -1.$ Note that $[-1]=0=\emptyset$ and $[0]=1=\{0\}.$ The \emph{simplex category} $\Delta$ is a full subcategory of $\Delta_+$ with objects $[n]$ for $n\geq 0.$ 

An augmented simplicial object $X$ of a category $\CC$ can be treated as a functor $X:\Delta_+^{\sf op}\to \CC,$ where  $X_n=X(n+1).$

The category $\Delta_+$ has a  strict monoidal structure $(\Delta_+,+,0)$ given by ``disjoint union'' $n\otimes m=n+m.$ There is a unique structure of monoid on $1\in \Delta_+$ defined by the maps $2\to 1$ and $0\to 1.$ This monoid is a universal monoid (or walking monoid) in the following sense. 
For any strict monoidal category $\CC=(\CC,I,\otimes)$ and any monoid $M=(M,\eta,\mu)$ in $ \CC$ there exists a unique strict monoidal functor $\Delta_+ \to  \CC$ that sends the monoid $1$ to the monoid $M$ (see 
\cite[Ch. VII, \S 5]{mclane1971categories}).

If we consider the opposite category $\Delta_+^{\sf op},$ then the object $1$ has a unique structure of comonoid, which is universal in the same sense. If $G$ is a comonoid in a strict monoidal category $\CC,$ then there exists a unique functor of strict monoidal categories  $\Delta_+^{\sf op}\to \CC$ that takes $1$ to $G.$ This functor is denoted by 
$${\sf Bar}(G):\Delta_+^{\sf op}\to \CC$$
and called the \emph{bar construction} of the comonoid $G.$ 

A \emph{comonad} $\GG$ on a category $\CC$ is a comonoid in the strict  monoidal category of endofunctors with respect to the composition of functors ${\sf End}(\CC)=({\sf End}(\CC),{\sf Id},\circ).$ Then each comonad $\GG$ defines an augmented  simplicial endofunctor ${\sf Bar}(\GG)$ (see \cite{barr1969homology}).  
$${\sf Bar}(\GG) : \hspace{1cm}
\begin{tikzcd} {\sf Id} & \arrow{l}[swap]{d_0} 
\GG  \arrow{r}{s_0} 
& \GG^2  \arrow[shift left=4mm]{l}[swap]{d_1} \arrow[shift left=-4mm]{l}[swap]{d_0} 
\arrow[shift left=4mm]{r}{s_0} \arrow[shift left=-4mm]{r}{s_1}
& \GG^3   \arrow[shift left=8mm]{l}[swap]{d_2} \arrow{l}[swap]{d_1} \arrow[shift left=-8mm]{l}[swap]{d_0} 
& \dots,
\end{tikzcd}
$$
It is easy to check that
 ${\sf Bar}(\GG)_n=\GG^{n+1}$ and 
$$ d_i=\GG^i \varepsilon  \GG^{n-i}: \GG^{n+1} \to \GG^n, \hspace{1cm} s_i=\GG^i \delta  \GG^{n-i}:\GG^{n+1} \to \GG^{n+2}.$$

\subsection{The bar construction for monads.}  The bar construction for monads is a particular case of the bar construction for comonads. Namely, each monad $\MM=(\MM,\eta,\mu)$ defines an adjunction $\FF:\CC \rightleftarrows \CC^\MM : \UU,$ where $\CC^\MM$ is the category of $\MM$-algebras, $\UU$ is the forgetful functor and $\FF$ is the functor of free $\MM$-algebra. Then  $\GG^\MM=\FF\UU$ is a comonad on the category $\CC^\MM$ and the bar construction of $\MM$ is ${\sf Bar}(\GG^\MM).$ 

\subsection{Extra degeneracy maps} Let $X$ be an augmented simplicial object in a category $\CC.$ {\it Extra degeneracy maps} for $X$ is a sequence of morphisms $s_{-1}:X_n\to X_{n+1}$ for $n\geq -1$ satisfying the extended simplicial identities 
\begin{itemize}
\item $d_0s_{-1}={\sf id};$
\item $d_is_{-1}=s_{-1}d_{i-1}$  for $i\geq 1;$
\item $s_{i}s_{-1}=s_{-1}s_{i-1}$ for $i\geq 0$ 
\end{itemize}
(see \cite{barr2019contractible}). The following lemma seems to be well known but we can not find the precise statement in the full generality in literature. 

\begin{lemma}[{cf. \cite[Th. 4.3]{barr2019contractible}}]
Let $X_\bullet$ be an augmented simplicial object in a category $\CC.$ If there exist extra degeneracy maps $s_{-1}:X_{n}\to X_{n+1},$ then $X_\bullet$ is strong homotopy equivalent to the constant simplicial object $X_{-1}^{\sf const}.$
\end{lemma}
\begin{proof}
Note that simplicial identities imply that for $i\geq j\geq 0,$  $i\geq 1$   we have 
\begin{equation}\label{eq_d_0}
  d_0^id_j=d_0^{i+1} \hspace{1cm} d_0^is_j=d_0^{i-1}.  
\end{equation}
Moreover, the identities for $s_{-1}$ imply that for for $i> j\geq 0,$  $i\geq 1$
\begin{equation}\label{eq_s_-1}
d_js_{-1}^i=s_{-1}^{i-1}, 
 \hspace{1cm} d_i s_{-1}^i=s_{-1}^i d_0,
 \hspace{1cm}
s_js_{-1}^i=s_{-1}^{i+1}.
\end{equation}
The identities  \eqref{eq_d_0} imply that there is a simplicial morphism  $f:X_\bullet\to X_{-1}^{\sf const}$ given by $f_n=d_0^{n+1}:X_n\to X_{-1}.$ 
The identities \eqref{eq_s_-1} imply that there is a simplicial morphism  $g:X_{-1}^{\sf const}\to X_\bullet$ given by $g_n=s_{-1}^{n+1}:X_{-1}\to X_n.$ 
It is easy to see that $fg={\sf id}.$ 
We claim that there is a homotopy  from ${\sf id}$ to $gf.$ We define the homotopy as $h_i=s_{-1}^{i+1}d_0^i:X_n\to X_{n+1}.$ 
In order to check that this map is a homotopy from ${\sf id}$ to $gf$, we have to prove the identities
$$
\begin{array}{ll}
 (1)\: d_0h_0={\sf id}; 
&
(2)\: d_{n+1}h_n=gf; 
\\
(3)\: d_ih_j=h_{j-1}d_i \text{\ \ if\ \ } i<j; 
&
(4)\: d_{j+1}h_{j+1}=d_{j+1}h_j; \\
(5)\: d_ih_j=h_jd_{i-1} \text{\ \ if \ \ } i>j+1;
& 
(6) s_ih_j=h_{j+1}s_i \text{\ \ if\ \ } i\leq j; 
\\
(7) s_ih_j=h_js_{i-1} \text{\ \  if\ \ } i>j. & 
\end{array}
$$
Let  us prove them.

(1) $d_0h_0=d_0s_{-1}={\sf id};$

(2) $d_{n+1}h_n=d_{n+1}s_{-1}^{n+1}d_0^n=s_{-1} d_n s_{-1}^{n} d_0^n = \dots = s_{-1}^{n+1} d_0^{n+1}.$ 

(3) $d_ih_j=d_i s_{-1}^{j+1}d_0^j=s_{-1}^{j} d_0^j=s_{-1}^j d_0^{j-1}d_i=h_{j-1}d_i;$ 

(4) $d_{j+1}h_{j+1}=d_{j+1} s_{-1}^{j+2}d_0^{j+1} = s_{-1}^{j+1}d_0^{j+1}= d_{j+1}s_{-1}^{j+1}d_0^j=d_{j+1} h_j;$

(5) $d_ih_j=d_is_{-1}^{j+1}d_0^j=s_{-1}^{j+1} d_{i-j-1} d_0^j = s_{-1}^{j+1} d_0^j d_{i-1} = h_jd_{i-1};$

(6) $s_ih_j=s_is_{-1}^{j+1} d_0^j = s_{-1}^{j+2} d_0^j= s_{-1}^{j+2} d_0^{j+1}s_i=h_{j+1}s_i$

(7) $s_ih_j=s_is_{-1}^{j+1}d_0^j=s_{-1}^{j+1}s_{i-j-1}d_0^j=s_{-1}^{j+1}d_0^js_{i-1}=h_js_{i-1}.$
\end{proof}

\subsection{The bar construction in the case of adjoint functors}

Let $\FF:\DD\leftrightarrows \CC: \UU$ be an adjunction with the unit $\eta:{\sf Id}_\DD\to \UU\FF$ and the counit $\varepsilon:\FF \UU\to {\sf Id}_\CC.$  Such an adjunction defines a comonad $(\GG, \varepsilon,\delta)$ on the category $\CC$ where $\GG=\FF\UU$ and $ \delta= \FF \eta\UU:\GG\to \GG^2.$ 

\begin{lemma}\label{lemma_extra} The augmented simplicial functor $\UU{\sf Bar}(\GG)\in s{\sf Funct}(\CC,\DD)$ has extra degeneracy maps given by 
$$s_{-1}=\eta \UU \GG^{n+1} : \UU\GG^{n+1}\to \UU\GG^{n+2}$$ 
for $n\geq -1.$ In particular, $\UU{\sf Bar}(\GG)$ is homotopy equivalent to the constant simplicial functor $\UU.$
\end{lemma}
\begin{proof} Note that 
 $ d_i=\UU\GG^i \varepsilon  \GG^{n-i},$ $ s_i=\UU\GG^i \FF\eta\UU  \GG^{n-i}$ for $i\geq 0.$ A straightforward computation using the formulas  $(\UU \varepsilon) \circ (\eta \UU)={\sf id}$ and $(\varepsilon \FF)\circ (\FF \eta)={\sf id}$ shows that $s_{-1}$ are extra degeneracy maps. 
\end{proof}

\section{Computing complexes for comonad derived functors}

During this section $\CC$ denotes a category $\GG,$ denotes a comonad on $\CC$, $\AA$ denotes an abelian category and 
$$\Phi:\CC\to \AA$$ 
is a functor. The comonad derived functors of $\Phi$ (or Barr-Beck derived functors)   $L^\GG_n\Phi:\CC\to \AA$ are defined as 
$$L^{\GG}_n\Phi = H_n(\Phi {\sf Bar}(\GG))$$
(see \cite{barr1969homology}).
Here we treat the simplcial object $\Phi {\sf Bar}(\GG)$ as a chain complex in the usual way. However straightforward computations by this definition can be complicated and in this section we present a general way for replacement of the complex $\Phi{\sf Bar}(\GG)$ by other complexes that we call (quasi-)computing complexes. Our theory of quasi-computing complexes echoes with the theory of acyclic models of Barr and Beck \cite{barr1966acyclic}. However our theory is independent and more convenient for our purposes. 

\subsection{Quasi-computing complexes}
A {\it quasi-computing complex} for the derived functors $L_*^{\GG}\Phi:\CC\to \AA$ is a functor to the category of non-negatively graded chain complexes
$$K_\bullet : \CC\to  {\sf Com}_{\geq 0}(\AA)$$
such that the following axioms are satisfied
\begin{itemize}
\item[(C1)]  $H_0(K_\bullet)\cong \Phi;$
\item[(C2)] $H_n(K_\bullet \GG )=0$ for $n\geq 1.$
\end{itemize}

\begin{theorem}\label{th_quasicomputing}
Let $K_\bullet$ be a quasi-computing complex for  $L_*^\GG\Phi.$ Then there is a first quadrant spectral sequence of homological type  
$$E^{1}_{i,j}=L^{\GG}_jK_i \Rightarrow L_{i+j}^{\GG}  \Phi.$$
Moreover, the differential on $E^1$ is induced by the differential on $K_\bullet.$
\end{theorem}
\begin{proof} Set $\BB={\sf Bar}(\GG)$ and $L_*=L^\GG_*.$ 
Consider a double chain complex $D_{i,j}=K_i(\BB_j).$ If we fix $i$ and consider the vertical homology we obtain by definition $H^{\sf vert}_j(D_{i,*})=L_jK_i.$ Then the first page of the spectral sequence $E$ of the double complex is $E^1_{i,j}=L_jK_i.$ 
On the other hand, if we fix $j$ and consider the horizontal homology,  we obtain $H^{\sf hor}_i(D_{*,j})=0$ for $ i\ne 0$ by axiom (C2);  and $H^{\sf hor}_0(D_{*,j})=\Phi\BB_j$ by the  axiom (C1). It follows that  the homology of the total complex of $D$ is $H_*(\Phi\BB )=L_*^\GG \Phi.$ 
\end{proof}

\subsection{Computing complexes}
Let $\FF:\DD\leftrightarrows \CC: \UU$ be an adjunction and $\GG=\FF\UU$ be the corresponding comonad on $\CC$. A computing complex for $L^\GG_*\Phi$  with respect to the adjunction $\FF:\DD\leftrightarrows \CC: \UU $ is a quasi-computing complex $K_\bullet$ for $L^\GG_*\Phi$  satisfying one more axiom: 
\begin{itemize}
\item[(C3)] for any $n\geq 0$ there exists a functor $\tilde K_n:\DD\to \AA$ such that  $K_n\cong \tilde K_n \UU.$
\end{itemize}

Further in this section we assume that $\FF:\DD\leftrightarrows \CC: \UU$ is an adjunction and $\GG=\FF\UU$ is the corresponding comonad. 

\begin{lemma}\label{lemma_computing}
Let $\Phi:\CC\to \AA$ be a functor such that there exists a functor $\tilde \Phi:\DD\to \AA$ such that $\Phi\cong \tilde{\Phi}\UU.$ Then $L_n^\GG\Phi =0$ for $n\ne 0$ and $L^\GG_0\Phi \cong \Phi.$
\end{lemma}
\begin{proof} By Lemma \ref{lemma_extra} $\UU{\sf Bar}(\GG)$ is homotopy equivalent to the constant simplicial functor $\UU.$ Hence $ \Phi {\sf Bar}(\GG)\cong \tilde{\Phi} \UU{\sf Bar}(\GG)$ is homotopy equivalent to the constant simplicial functor $\tilde \Phi \UU\cong \Phi.$
\end{proof}

\begin{theorem}\label{th_computing}
Let $K_\bullet$ be a computing complex for  $L^\GG_*\Phi.$ Then there is a natural isomorphism 
$$H_*(K_\bullet)\cong L^\GG_*\Phi.$$
\end{theorem}
\begin{proof}
It follows from Lemma \ref{lemma_computing} and Theorem \ref{th_quasicomputing}.
\end{proof}

\subsection{Examples of computing complexes}

In this section we consider examples of computing complexes for homology of various algebraic systems. Further we will often use the following construction. If $K_\bullet\in {\sf Com}_{\geq 0}(\AA)$ is a non-negatively graded chain complex in some abelian category $\AA$, we denote by $K'_\bullet$ its shifted and truncated version  
\begin{equation}\label{eq_K'}
(K_n')=
\begin{cases}
K_{n+1} & n> 0\\
{\sf Ker}(K_1\to K_0) & n=0\\
0 & n<0.
\end{cases}
\end{equation}

\subsubsection{Group homology.} Let $\CC={\sf Gr}$ be the category of groups, $\DD={\sf Sets}$ be the category of sets, $\AA={\sf Ab}$ be the category of abelian groups and $\FF: {\sf Sets} \leftrightarrows   {\sf Gr}:\UU$ be the standard free-forgetful adjunction and $\Phi={\sf ab}:{\sf Gr}\to {\sf Ab}$ is the functor of group abelianization. Consider the standard complex for group homology $C_\bullet(G),$ where 
$C_n(G)=\ZZ[G^n]$ and 
$$\partial_n(g_1,\dots,g_n)=(g_2,\dots,g_n)+ \sum_{i=1}^{n-1} (-1)^i (g_1,\dots,g_ig_{i+1},\dots,g_n) +(-1)^n (g_1,\dots,g_{n-1})$$ and its shifted and truncated version   
$C'_\bullet(G)$ (see \eqref{eq_K'}). Then $C'_\bullet(G)$
is a computing complex for group homology $L_*^\GG {\sf ab}$.  Indeed (C1) is satisfied because $H_0(C'_\bullet(G))=H_1(G)=G_{\sf ab};$ (C2) is satisfied because higher homology of a free group is trivial; (C3) is satisfied because the functors $\ZZ[G^n]$ do not depend on the group structure of $G$ but only on the underlying set (the differential $\partial$ depend on the group structure). Then Theorem \ref{th_computing} implies the well-known fact that the group homology coinsides with the 
comonad derived functors of the functor of abelianization \cite{quillen2006homotopical}
\begin{equation}
H_{*+1}(G)\cong L^\GG_*{\sf ab}(G).   
\end{equation}

\subsubsection{Homology of Lie algebras over fields} Let $\kk$ be a field, $\CC={\sf Lie}$ be the category of Lie algebras over $\kk,$ $\DD={\sf Vect}$ be the category of $\kk$-vector spaces, $\FF^{\sf V}: {\sf Vect} \leftrightarrows  {\sf Lie}:\UU^{\sf V}$ be the free-forgetful adjunction, $\AA={\sf Vect}$ be the category of $\kk$-vector spaces and $\Phi={\sf ab}:{\sf Lie}\to {\sf Vect}$ be the functor of abelianization. We denote by $\GG^{\sf V}:=\FF^{\sf V} \UU^{\sf V}$ the corresponding  comonad on the category of Lie algebras.

Consider the Chevalley-Eilenberg complex ${\sf CE}_\bullet(\gg)$ for a Lie algebra $\gg,$ where ${\sf CE}_n(\gg)=\Lambda^{n}\gg$  and 
$$\partial(x_1\wedge \dots \wedge x_{n} ) = \sum_{i<j} (-1)^{i+j} [x_i,x_j] \wedge x_1\wedge \dots \wedge \hat x_i \wedge \dots \wedge \hat x_j \wedge \dots \wedge x_n.$$ Denote by  ${\sf CE}'_\bullet(\gg)$ its shifted and truncated version ${\sf CE}(\gg)$ (see \eqref{eq_K'}).

\begin{proposition}\label{prop_computing_CE}
 ${\sf CE}'_\bullet(\gg)$ is a computing complex for $L_*^{\GG^{\sf V}} {\sf ab}.$
\end{proposition}
\begin{proof}
(C1) is satisfied because $H_0({\sf CE}'_\bullet(\gg))=H_1(\gg)=\gg_{\sf ab}.$ (C2) is satisfied because in the case of a Lie algebra over a field the Lie algebra of the form $\GG^{\sf V}(\gg)$ is free and $H_n({\sf CE}_\bullet(\ff))=H_n(\ff)=0$ for any free Lie algebra $\ff$. (C3) is satisfied because $\Lambda^n \gg$ depends only on the structure of vector space on $\gg.$
\end{proof}

Proposition \ref{prop_computing_CE} and Theorem \ref{th_computing} imply the following well-known statement.
\begin{corollary} For any Lie algebra $\gg$ over a field there is a natural isomorphism
$$H_{*+1}(\gg)\cong L^{\GG^{\sf V}}_*{\sf ab}.$$
\end{corollary}

\subsubsection{Homology of Leibniz algebras over fields} A Leibniz algebra is a generalization of a Lie algebra introduced in \cite{loday1993universal}.
A (left) \emph{Leibniz algebra} over a field $\kk$ is defined as  a vector space $\ll$ together with a bilinear operation $[-,-]$ satisfying the Leibniz identity $[a,[b,c]]=[[a,b],c] - [[a,c],b].$
In the theory of Leibniz algebras there is an analogue of the Chevalley-Eilenberg complex, that we call Leibniz-Chevalley-Eilenberg complex, where the exterior powers are replaced by the  tensor powers 
$${\sf LCE_\bullet(\ll)} : \hspace{1cm} \dots \to  \ll^{\otimes 3} \to \ll^{\otimes 2} \to \ll \to \kk \to 0.$$
The homology of a Leibniz algebra is defined as 
$HL_*(\ll)=H_*({\sf LCE}_\bullet(\ll)).$
It is known that for a free Leibniz algebra $\ff$ its higher homology vanishes \cite[Cor. 3.5]{loday1993universal}
$$HL_n(\ff)=0, \hspace{1cm} n\geq 2.$$
We denote by ${\sf LCE}'_\bullet(\ll)$ the shifted and truncated version of ${\sf LCE}_\bullet(\ll)$ (see \eqref{eq_K'}). Consider the free-forgetful adjunction $\FF^{\sf V}:{\sf Vect} \rightleftarrows {\sf Leibniz} : \UU^{\sf V}$ and the corresponding comonad $\GG^{\sf V}.$ Similarly to the case of Lie algebras one can prove the following.
\begin{proposition}
${\sf LCE}'_\bullet(\ll)$ is a computing complex for 
$L_*^{\GG^{\sf V}} {\sf ab}: {\sf Leibniz}\to {\sf Vect}.$
\end{proposition}
\begin{corollary} For any Leibniz algebra $\ll$ over a field $\kk$ there is a natural isomorphism
$$HL_{*+1}(\ll)\cong L^{\GG^{\sf V}}_*{\sf ab}.$$
\end{corollary}

\subsubsection{Rack homology} A \emph{rack} is a magma  $R=(R,\triangleright)$  such that the left multiplication with any element is its automorphism. The abelianization of a rack $R$ 
is defined as the quotient of the free abelian group $\ZZ[R]$ generated by $R$ by the subgroup generated by elements of the form $y-x\triangleright y$
$$R_{\sf ab}:=\ZZ[R]/\langle y-x\triangleright y \rangle.$$ Following \cite{fenn1995trunks}, \cite{fenn2007rack} (see also   \cite{szymik2019quandle}) consider the complex $CR_\bullet(R)$ such that $CR_n(R)=\ZZ[R^n]$ is a free abelian group generated by $R^n$ and 
$$\partial(x_1,\dots,x_n)=\sum_{i=1}^n (-1)^i ((x_1,\dots,\hat x_i,\dots,x_n) - (x_1,\dots,x_{i-1},x_i \triangleright x_{i+1}, \dots,x_i \triangleright x_n)).$$
Then the rack homology is defined as $HR_*(R):=H_*(CR_\bullet(R)).$ It is known that the higher rack homology of a free rack is trivial \cite{farinati2014homology},  \cite[Th.5.13]{fenn2007rack}. 

Let $\CC={\sf Racks}$ be the category of racks, $\DD={\sf Sets}$ be the category of sets,  $\FF:{\sf Sets}\leftrightarrows {\sf Racks}:\UU$ be the free-forgetful adjunction, $\AA={\sf Ab}$ be the category of abelian groups and $\Phi={\sf ab}:{\sf Racks}\to {\sf Ab}$ be the functor of abelianization.
Consider the shifted and truncated version ${CR}'_\bullet(R)$ of the complex ${CR}_\bullet(R)$ (see \eqref{eq_K'}).

\begin{lemma}\label{lemma_rack}
${CR}'_\bullet(R)$ is a computing complex for derived functors of the rack abelianization functor 
$$L^\GG_*{\sf ab}:{\sf Racks}\longrightarrow {\sf Ab}.$$
\end{lemma}
\begin{proof}
(C1) is satisfied because $H_0(CR'_\bullet(R))=HR_1(R)=R_{\sf ab};$ (C2) is satisfied because the higher homology of a free rack is trivial; (C3) is satisfied because $CR'_n(R)=\ZZ[R^{n+1}]$ depends only on the underlying set of $R.$
\end{proof}

\begin{proposition}
The rack homology is isomorphic to the comonad derived functors of the rack abelianization functor ${\sf ab}:{\sf Racks}\to {\sf Ab}:$
$$HR_{*+1}(R)\cong L^\GG_*{\sf ab}(R).$$
\end{proposition}
\begin{proof} It follows from Lemma \ref{lemma_rack} and
Theorem \ref{th_computing}.
\end{proof}

 \section{Chevalley-Eilenberg spectral sequence} 
 
\subsection{The spectral sequence}

Let  $\kk$ be a commutative ring, ${\sf Mod}$ denote the category of $\kk$-modules and ${\sf Lie}$ denote the category of Lie algebras over $\kk$. Consider two free-forgetful adjunctions $$\FF^{\sf S}:{\sf Sets} \leftrightarrows {\sf Lie}:  \UU^{\sf S} \hspace{0.5cm} \text{and} \hspace{0.5cm} \FF^{\sf M}:{\sf Mod} \leftrightarrows {\sf Lie}:  \UU^{\sf M}$$
and the corresponding comonads $\GG^{\sf S}:=\FF^{\sf S}\UU^{\sf S}$ and $\GG^{\sf M}:=\FF^{\sf M}\UU^{\sf M}$ on the category ${\sf Lie}.$ 
In this section we discuss a connection between simplicial homology of Lie algebras  $H_{*+1}^{\sf S}=L_*^{\GG^{\sf S}}{\sf ab}$ and the Chevalley-Eilenberg complex. 

 For a (not necessarily  additive) functor $\Psi:{\sf Mod}\to {\sf Mod}$ we denote by $L_n\Phi:{\sf Mod}\to {\sf Mod}$ its non-shifted derived functor in the sense of Dold-Puppe \cite{dold1961homologie}
$$L_n\Psi (A) := L^{\sf Dold-Puppe}_n\Psi (A,0).$$ 
For a Lie algebra $\gg$ for simplicity we set $L_n\Psi (\gg):=(L_n\Psi)( \UU^{\sf M}\gg).$  

\begin{lemma}\label{lemma_derived_DP}
Let $\Psi:{\sf Mod}\to {\sf Mod}$  be a functor. Then the derived functor of the composition $\Psi \UU^{\sf M}: {\sf Lie}\to {\sf Mod}$ with respect to the comonad $\GG^{\sf M}$ is isomorphic to the composition of the derived functor of $\Psi$ in the sense of Dold-Puppe with $\UU^{\sf M}:{\sf Lie}\to {\sf Mod}:$
$$L^{\GG^{\sf S}}_* (\Psi \UU^{\sf M} ) \cong  (L_*\Psi) \UU^{\sf M}.$$
\end{lemma}
\begin{proof}
Set $\ff_\bullet:={\sf Bar}(\GG^{\sf S})(\gg).$ Then $\ff_n$ is a free Lie algebra of a free $\kk$-module  for any $n.$ 
Hence  $\UU^{\sf M} \ff_n $ is a free $\kk$-module \cite[Cor.0.10]{reutenauer2003free}. Therefore $\UU^{\sf M}\ff_\bullet$ is a free simplicial resolution of $\UU^{\sf M} \gg.$ 
Then by definition of the Dold-Puppe derived functor we have $(L_*\Psi)(\UU^{\sf M}\gg)=  \pi_*(\Psi (\UU^{\sf M} \ff_\bullet))=  L^{\GG^{\sf S}}_* (\Psi \UU^{\sf M} ).$
\end{proof}

Note that components of the Chevalley-Eilenberg complex can be functorially described as  ${\sf CE}_n=\Lambda^n \UU^{\sf M}:{\sf Lie}\to {\sf Mod}$ and  $\partial_n :{\sf CE}_n\to {\sf CE}_{n-1}$ is a natural transformation. It induces a natural transformation $L^{\GG^{\sf S}}_m \Lambda^n \UU^{\sf M} \to  L^{\GG^{\sf S}}_m \Lambda^{n-1} \UU^{\sf M}.$ Using Lemma \ref{lemma_derived_DP} we obtain a map 
$$\partial_{n,m} : L_m \Lambda^n \UU^{\sf M} \to  L_m \Lambda^{n-1} \UU^{\sf M}.$$ 
Then for any $m\geq 0$ and any $\gg$  the Chevalley-Eilenberg complex induces a complex  
$$L_m{\sf CE}_\bullet(\gg): \hspace{1cm}  \dots\longrightarrow  L_m \Lambda^n (\gg) \overset{\partial_{n,m}}\longrightarrow  L_m \Lambda^{n-1}(\gg) \longrightarrow \dots  .$$ 
The homology of this complex will be denoted by 
$$H_{n,m}^{\sf CE}(\gg):=H_n(L_m{\sf CE}_\bullet (\gg)).$$

\begin{theorem}[Chevalley-Eilenberg spectral sequence]\label{th_CE}
For any Lie algebra $\gg$ over a commutative ring $\kk$ there is a natural spectral sequence $E=E(\gg)$ of homological type such that 
$$E^1_{ij}=L_i \Lambda^j (\gg)  \Rightarrow H_{i+j}^{\sf S}(\gg).$$
Moreover, the differential of $E^1$ comes from  the differential of $L_i{\sf CE}_\bullet(\gg),$
and hence, the second page consists of $H_{j,i}^{\sf CE}(\gg)$ 
$$E^2_{i,j}=H_{j,i}^{\sf CE}(\gg)\Rightarrow H_{i+j}^{\sf S}(\gg).$$
\end{theorem}
\begin{proof}
Consider the shifted and truncated version of the Chevalley-Eilenberg complex ${\sf CE'}_\bullet(\gg).$ We claim that it is a quasi-computing complex for  $L_*^{\GG^{\sf S}}{\sf ab}.$ Indeed (C1) is satisfied because $H_0({\sf CE}'_\bullet(\gg))=H_1({\sf CE}_\bullet(\gg))={\sf ab}(\gg);$ (C2) is satisfied because for any Lie algebra which is free as a $\kk$-module \cite[Ch.XII, Th.7.1]{cartan1999homological} the homology of the  Chevalley-Eilenberg complex coincides with the Tor homology, and hence, $H_n({\sf CE}'_\bullet(\ff))=0$ for $n>0$ and any free Lie algebra of a free $\kk$-module $\ff.$ 
Then the assertion follows from Theorem \ref{th_quasicomputing}, Lemma \ref{lemma_derived_DP} and the formula ${\sf CE}_n=\Lambda^n\UU^{\sf M}$.
\end{proof}

\begin{remark}\label{rem_CE_spec}
$\Lambda^0(A)=\kk$ is a constant functor and $\Lambda^1(A)=A$ is an identity functor. It follows that 
$$L_i \Lambda^0(\gg)=0, \hspace{1cm} L_i\Lambda^1(\gg)=0,\hspace{1cm} i>0,$$
and hence, 
$$H_{0,i}^{\sf CE}(\gg)=0, \hspace{1cm} H_{1,i}^{\sf CE}(\gg)=0, \hspace{1cm} i>0.$$
It follows that there is an exact sequence
$$L_1\Lambda^3(\gg) \overset{\partial_{3,1}}\longrightarrow L_1 \Lambda^2 (\gg) \longrightarrow  H_{2,1}^{\sf CE}(\gg)\longrightarrow 0,$$
where $\partial_{3,1}$ is the differential in $L_1{\sf CE}_\bullet(\gg).$
\end{remark}

\begin{corollary}\label{cor_H_2S}
For any Lie algebra $\gg$ over any commutative ring $\kk$ there is an isomorphism 
$$H_2^{\sf S}(\gg)\cong H_2^{\sf CE}(\gg).$$
\end{corollary} 

\begin{corollary}\label{cor_H_3S}
For any Lie algebra $\gg$ over any commutative ring $\kk$ there is an exact sequence for $H_3^{\sf S}(\gg)$:
$$ H_4^{\sf CE}(\gg) \to  H_{2,1}^{\sf CE}(\gg) \to H_3^{\sf S}(\gg)\to H_3^{\sf CE}(\gg) \to 0. $$
\end{corollary}

The aim of the rest of this section is to give a more detailed description of the functor $H_{2,1}^{\sf CE}(\gg)$ and the map $\partial_{3,1} : L_1\Lambda^3 \to L_1 \Lambda^2$ in the spectral sequence in the case $\kk=\ZZ$ (Theorem \ref{th_H21}). We use descriptions of the functors $L_1\Lambda^2$ and $L_1 \Lambda^3$ that can be found in \cite{breen1999functorial}, \cite{breen2011derived}, \cite{jeanphdthesis} and some additional work for element-wise computations with this functors.

\subsection{First derived functors of some elementary functors}
 
In this subsection we follow the notations of Breen's paper \cite{breen1999functorial} and Jean's thesis \cite{jeanphdthesis} and intensively use some of their results. Further in this section we assume that $\kk=\ZZ.$ 

For any $n\geq 1$ and an abelian group $A$ we denote by  $${}_nA=\{a\in A\mid na=0\}$$
its $n$-torsion subgroup. For any abelian groups $A,B$ Breen defines a natural map 
$$\tau_n:{}_nA\otimes {}_nB \longrightarrow {\sf Tor}(A,B)$$
such that the group ${\sf Tor}(A,B)$ is generated by the elements $\tau_n(a,b)$ for all $n$ (\cite[Prop. 3.5]{breen1999functorial}). Moreover, the group ${\sf Tor}(A,B)$
can be presented as an abelian group with generators $\tau_n(a\otimes b), a,b\in {}_nA$ and relations 
$$\tau_{nm}(a\otimes b)=\tau_n(ma\otimes b)=\tau_n(a\otimes m b)$$
for $a,b\in {}_{nm}A.$
Breen  gives a conceptual definition of this homomorphism using the language of the derived category $D({\sf Ab})$ but for our purposes we prefer to use an  explicit language of resolutions.  
Let $P_\bullet \to A$ be a flat (chain) resolution of $A.$ For any $a\in A$ we choose  its preimage $v(a)\in P_0$ and for $a\in {}_nA$  we denote by  $u(a)\in P_1$ an element such that $\partial(u(a))=nv(a).$ Then  $\tau_n(a\otimes b)$ corresponds to the class of $u(a)\otimes b\in P_1\otimes B$ in $H_1(P_\bullet \otimes B)\cong {\sf Tor}(A,B).$
\begin{equation}
\tau_n(a\otimes b) \leftrightarrow [u(a)\otimes b] \in H_1( P_\bullet \otimes B).    
\end{equation}
It is easy to check that this definition is equivalent to the definition of Breen \cite[(3.6)]{breen1999functorial}.

More symmetrically, if $Q_\bullet \to B$ is a flat resolution of $B,$ then $\tau_n(a\otimes b)$ corresponds to the class of $u(a)\otimes v(b)-v(a)\otimes u(b)\in (P_1\otimes Q_0)\oplus (P_0\otimes Q_1) $ in $H_1(P_\bullet\otimes  Q_\bullet).$ Indeed, it is a cycle and the quasi-isomorphism $P_\bullet \otimes Q_\bullet \to P_\bullet \otimes B$ 
sends it to $u(a)\otimes b.$
\begin{equation}
  \tau_n(a\otimes b) \leftrightarrow  [u(a)\otimes  v(b) -  v(a)\otimes u(b)]\in H_1(P_\bullet\otimes Q_\bullet).  
\end{equation}

Consider the functor of the  tensor square $$\otimes^2:{\sf Ab}\to {\sf Ab}.$$ The Eilenberg-Zilber theorem implies that its first derived functor is isomorphic to the Tor functor
$$  
L_1{\otimes}^2(A) \cong {\sf Tor}(A,A)
$$ 
The composition of $\tau_n$ with this isomorphism will be denoted by
$$\tilde \tau_n : {}_nA \otimes {}_nA \longrightarrow L_1{\otimes}^2(A).$$
Now take a {\it simplicial} flat resolution  ${\bf P}_\bullet\to A.$ For $a\in A$ we fix some its preimage ${\bf v}(a)\in {\bf P}_0$  and for $a\in {}_nA$ we denote by ${\bf u}(a)\in  {\bf P}_1$  an element such that $d_1({\bf u}(a))=0$ and $d_0({\bf u}(a))=n{\bf v}(a).$

\begin{lemma}\label{lemma_L_1tensor} For any flat simplicial resolution ${\bf P}_\bullet\to A$ and $a,b\in {}_nA$
the element $\tilde \tau_n(a\otimes b) \in L_1{\otimes}^2(A)$ corresponds to the class of the element $$s_0{\bf v}(a)\otimes {\bf u}(b)-{\bf u}(a)\otimes s_0{\bf v}(b)$$ in $\pi_1(\otimes^2 {\bf P}_\bullet).$ 
\end{lemma}
\begin{proof}
Consider the Moore complex of the simplicial resolution $P_\bullet=N({\bf P}_\bullet).$ Then $P_\bullet$ is a flat (chain) resolution of $A$ and we can choose $v(a)$ and $u(a)$ such that the embedding $P_i\hookrightarrow {\bf P}_i$ sends ${v}(a)$ to ${\bf v}(a)$ and $u(a)$ to ${\bf u}(a).$ The Eilenber-Zilber map $P_\bullet \otimes P_\bullet \to N( {\bf P}_\bullet \otimes {\bf P}_\bullet)$ in this dimension acts as follows
$$P_0\otimes P_1 \to N({\bf P}_\bullet\otimes {\bf P}_\bullet)_1, \hspace{1cm} p\otimes q \mapsto s_0p\otimes q,$$
$$P_1\otimes P_0 \to N({\bf P}_\bullet\otimes {\bf P}_\bullet)_1, \hspace{1cm} q\otimes p \mapsto q\otimes s_0p.$$
The assertion follows. 
\end{proof}

Following Breen we denote by $\Omega$ the first derived functor of the functor of exterior square  $\Lambda^2:{\sf Ab}\to {\sf Ab}$ 
$$\Omega A:=L_{1}\Lambda^2(A).$$ The natural transformation $\Lambda^2 A \to {\otimes}^2 A,$ given by $a\wedge b \mapsto a\otimes b - b\otimes a$ induces an embedding \cite[Th.2.3.3]{jeanphdthesis}, \cite[\S 2.2]{breen2011derived} 
$$\Omega A \mono L^1{\otimes}^2 (A ) \cong {\sf Tor}(A,A).$$ 
Moreover, there is a natural action of the symmetric group $\Sigma_2$ on ${\sf Tor}(A,A)$ such that $\Omega A$ is isomorphic to the group of invariants of this action 
$$\Omega A\cong {\sf Tor}(A,A)^{\Sigma_2}.$$ 
The morphism $\tau_n :{}_nA\otimes {}_nA\to {\sf Tor}(A,A)$ respects the action of $\Sigma_2$ and induces a homomorphism from the divided square of ${}_nA$
$$\lambda_n:\Gamma^2({}_nA)\to \Omega A$$
such that the diagram
$$
\begin{tikzcd}
\Gamma^2({}_nA)\arrow{rr}{\lambda_n}
\arrow{d} & & \Omega A\arrow[rightarrowtail]{d} \\
{}_n A \otimes {}_n A\arrow{rr}{\tau_n} & & {\sf Tor}(A,A)
\end{tikzcd}
$$
is commutative \cite[(5.6)]{breen1999functorial}. 
The abelian group $\Omega A$ is generated by elements of the form  $$\omega_{2}^n(a):=\lambda_n(\gamma_{2}(a)),\hspace{1cm} a\in {}_nA.$$
We also consider the following elements in $\Omega A$ $$\omega^n_1(a)*\omega^n_1(b)
:=\omega_2^n(a+b)-\omega_2^n(a)-\omega_2^n(b)=\lambda_n(\gamma_1(a)\gamma_1(b)).$$

\begin{lemma}\label{lemma_L1exteriorSquare}
Let $a\in {}_nA$ and ${\bf P}_\bullet$ be a flat simplicial resolution of $A.$  Then the element $\omega^n_2(a)\in \Omega A$ corresponds to the class of  $$s_0 {\bf v}(a)\wedge {\bf u}(a)$$ in $\pi_1(\Lambda^2 {\bf P}_\bullet).$ The element $\omega_1^n(a)*\omega_1^n(b)$ corresponds to the class of 
$$s_0{\bf v}(a) \wedge {\bf u}(b) + s_0{\bf v}(b)\wedge {\bf u}(a).$$
\end{lemma}
\begin{proof}
The map $\Lambda^2 {\bf P}_\bullet \to {\otimes}^2{\bf P}_\bullet$ sends $s_0 {\bf v}(a)\wedge {\bf u}(a)$ to the element $s_0 {\bf v}(a)\otimes {\bf u}(a) - {\bf u}(a)\otimes s_0 {\bf v}(a)$ that corresponds to $\tau_n(a\otimes a)$ by Lemma \ref{lemma_L_1tensor}. Since the image of $\gamma_2(a)$ in $ {}_nA\otimes {}_nA$ is $a\otimes a,$ we obtain that $s_0 {\bf v}(a) \wedge {\bf u}(a)$ corresponds to $\omega_2^n(a)$. 

Note that the class of the element $s_0 {\bf v}(a) \wedge {\bf u}(a)$ does not depend on the choice of ${\bf v}$ and ${\bf u}.$ We can choose new preimages families of ${\bf v}',$ ${\bf u}'$ such that for two fixed elements $a,b\in {}_nA$ we have ${\bf v}'(a+b)={\bf v}(a)+{\bf v}(b)$ and   ${\bf u}'(a+b)={\bf u}(a)+{\bf u}(b).$ Then the element $\omega_1^n(a)*\omega_1^n(b)=\omega_2^n(a+b) - \omega_2^n(a) - \omega_2^n(b)$ corresponds to the class of  
$$s_0{\bf v}'(a+b)\otimes {\bf u}'(a+b) - s_0 {\bf v}(a)\otimes {\bf u}(a) - s_0 {\bf v}(b)\otimes {\bf u}(b).$$
The assertion follows.
\end{proof}

The multiplication map $\Lambda^sA\otimes \Lambda^tA\to \Lambda^{s+t}A$ induces a map in the derived category 
$L\Lambda^s(A) \otimes^{L} L\Lambda^t(A) \longrightarrow L\Lambda^{s+t}(A).$ The K\"uneth formula then gives a map
$L_1\Lambda^2 (A) \otimes L_0 \Lambda^1 (A) \to L_1 \Lambda^3(A) $ that can be rewritten as 
$$\alpha: \Omega A\otimes A \epi  L_1\Lambda^3(A) $$
Breen proved that the map is an epimorphism and give some more detailed description of $L_1\Lambda^3$ in these terms and Jean generalized this for $L_i\Lambda^n.$ 
We set $$\alpha_n(a,b):=\alpha(\omega^n_2(a)\otimes b)$$ for $a\in {}_nA$ and $b\in A.$ The elements of the form $\alpha_n(a,b)$ generate $L_1\Lambda^3(A)$ \cite[Prop.6.15]{breen1999functorial}. 

\begin{lemma} \label{lemma_alpha_n}
For any simplicial resolution ${\bf P}_\bullet$ of an abelian group $A$ an element $\alpha_n(a,b)\in L_1\Lambda^3(A)$ corresponds to the class of $$s_0{\bf v}(a)\wedge {\bf u}(a) \wedge s_0{\bf v}(b)$$ in $\pi_1(\Lambda^3 {\bf P}_\bullet ).$ 
\end{lemma}
\begin{proof} The K\"unneth map in the simplicial case 
$\pi_1(\Lambda^2 {\bf P}_\bullet )\otimes \pi_0({\bf P} ) \to \pi_1( \Lambda^2 {\bf P}_\bullet \otimes {\bf P}_\bullet ) $ 
is the composition of the chain K\"unneth map and the Eilenberg-Zilber map
$$N(\Lambda^2 {\bf P}_\bullet)_1\otimes N({\bf P}_\bullet)_0 \to N(\Lambda^2 {\bf P}_\bullet \otimes {\bf P}_\bullet  )_1 $$
which  in this dimension is given by 
$$ \theta \otimes p \mapsto \theta \otimes s_0p.$$
By Lemma \ref{lemma_L1exteriorSquare} the element $\omega_2^n(a)$ corresponds to the class of the element $s_0{\bf v}(a) \wedge {\bf u}(a)$ in $\pi_1( \Lambda^2 {\bf P}_\bullet ).$ The element $b\in A$ corresponds to ${\bf v}(a)$ in $\pi_0( {\bf P}_\bullet).$ Therefore the Eilenberg-Zilber map sends the element corresponding to $\omega_2^n(a)\otimes b$ to $s_0{\bf v}(a)\wedge {\bf u}(a) \otimes s_0{\bf v}(b).$ The multiplication  map sends it to  $s_0{\bf v}(a)\wedge {\bf u}(a) \wedge s_0{\bf v}(b).$
\end{proof}

\subsection{A description of $H_{2,1}^{\sf CE}(\gg)$} 

\begin{theorem}\label{th_H21}
For any Lie algebra $\gg$ over $\ZZ$ there is an  exact sequence
\begin{equation}\label{eq_H21CE}
\Omega \gg \otimes \gg \overset{\varphi}\longrightarrow \Omega \gg \longrightarrow H_{2,1}^{\sf CE}(\gg) \longrightarrow 0,    
\end{equation}
where 
\begin{equation}\label{eq_phi}
\varphi(\omega_2^n(a)\otimes b )= \omega_1^n([a,b])*\omega^n_1(a).    
\end{equation}
\end{theorem}
\begin{proof} By Remark \ref{rem_CE_spec} we have an exact sequence $$L_1\Lambda^3(\gg) \overset{\partial_{3,1}}\longrightarrow \Omega \gg \longrightarrow H_{2,1}^{\sf CE}(\gg)\to 0.$$ We define $\varphi$ as the composition of $\alpha:\Omega \gg\otimes \gg\epi L_1\Lambda^3 \gg$ and $\partial_{3,1}.$ Since $\alpha$ is an epimorphism, we obtain an exact sequence \eqref{eq_H21CE}. Then we only need to prove the formula \eqref{eq_phi}. 

Note that $\varphi(\omega^n_2(a)\otimes b)=\partial_{3,1}(\alpha_n(a,b)).$ Then we need to prove the formula 
$$\partial_{3,1}(\alpha_n(a,b)) = \omega_1^n([a,b])*\omega^n_1(a).$$ 
Set $\ff_\bullet={\sf Bar}(\GG^{\sf S})(\gg).$ Hence $\UU^{\sf M}\ff_\bullet $ is a flat resolution of $\UU^{\sf M}\gg.$ By Lemma \ref{lemma_alpha_n} the element $\alpha_n(a,b)$ can be presented by the element 
\begin{equation}\label{eq_alpha_n_2}
s_0{\bf v}(a) \wedge {\bf u}(a) \wedge s_0{\bf v}(b) \in \Lambda^3 \ff_\bullet.   
\end{equation}
The map $\partial_{3,1}$ is induced by the map $\partial_3:\Lambda^3 \gg \to \Lambda^2 \gg$ which is given by 
$$\partial_3( x_1\wedge x_2 \wedge x_3)= - [x_1,x_2] \wedge x_3 +[x_1,x_3] \wedge x_2 - [x_2,x_3] \wedge x_1.$$ 
Then the element 
\eqref{eq_alpha_n_2} 
is mapped to 
\begin{equation}\label{eq_sum_comp}
    - [s_0{\bf v}(a),{\bf u}(a) ] \wedge s_0 {\bf v}(b) +[s_0{\bf v}(a),s_0{\bf v}(b)] \wedge {\bf u}(a) - [{\bf u}(a),s_0{\bf v}(b)] \wedge s_0{\bf v}(a)
\end{equation}

Note that for two elements $\theta, \eta \in \ff_1$ such that $\theta\in Z_1(\ff_\bullet)$  we have that $\theta \wedge \eta \in \Lambda^2 \ff_\bullet$ is a Moore boundary. Indeed, since $Z_1(\ff_\bullet)=B_1(\ff_\bullet),$ there is $\zeta\in N_2(\ff_\bullet) $ such that $d_0(\zeta)=\theta.$ Hence $\zeta\wedge s_0\eta \in N_2(\Lambda^2 \ff_\bullet )$ and $d_0( \zeta \wedge s_0\eta)=\theta \wedge \eta.$  It follows that for any $\theta_1,\theta_2,\eta \in \ff_1,$ if $d_0(\theta_1)=d_0(\theta_2)$ and $d_1(\theta_1)=d_1(\theta_2)=0,$ the elements $$\theta_1 \wedge \eta \hspace{1cm} \text{and} \hspace{1cm} \theta_2 \wedge \eta  $$
differ by a Moore boundary and $\theta_1\wedge \eta$ can be replaced by $\theta_2\wedge \eta$ in any computation up to boundary.  Using this, the first summand  of \eqref{eq_sum_comp} can be replaced by zero, because $d_i([s_0{\bf v}(a),{\bf u}(a)] )=0 $ for $i=0,1;$ the second summand can be replaced by $ s_0{\bf v}([a,b]) \wedge {\bf u}(a);$ and the last summand can be replaced by $-{\bf u}([a,b])\wedge s_0{\bf v}(a).$ Therefore the element $\partial_{3,1}(\alpha_n(a,b))$ corresponds to 
$$s_0{\bf v}([a,b])\wedge {\bf u}(a) + s_0{\bf v}(a) \wedge {\bf u}([a,b]).$$
The assertion follows from Lemma  \ref{lemma_L1exteriorSquare}.
\end{proof}

\subsection{Example}

\begin{proposition}\label{proposition_main_example}
Let $\kk=\ZZ$ and $\ff(x,y)$ be a free Lie algebra over $\ZZ$ generated by two letters. Consider the tensor product 
$$\gg:=\ff(x,y)\otimes \ZZ/2$$
as a Lie algebra over $\ZZ.$ Then
$$H_3^{\sf S}(\gg)\cong  (\ZZ/2)^{\oplus \infty}, \hspace{1cm} H_3^{\sf T}(\gg)\cong (\ZZ/2)^4, \hspace{1cm} H_3^{\sf CE}(\gg)=0,$$
where $(-)^{\oplus \infty}$ denotes a countable direct sum.
Moreover, $H^{\sf CE}_{2,1}(\gg ) \cong (\ZZ/2)^{\oplus \infty},$ $H^{\sf T}_{2n+2}(\gg)=0$ and $H^{\sf T}_{2n+1}(\gg)=(\ZZ/2)^{2^{n+1}}$ for $n\geq 0.$ 
\end{proposition}
\begin{proof}
If we consider $\gg$ as a Lie algebra over $\ZZ/2,$ then it is just the free algebra over the field  $\ZZ/2.$ The Chevalley-Eilenberg complex ${\sf CE}_\bullet(\gg)$ over $\ZZ$ equals to the Chevalley-Eilenberg complex over $\ZZ/2$ (that computes the ordinary homology over a field) and hence, $H_n^{\sf CE}(\gg)=0$ for any $n\geq 2.$ 

Prove that $H^{\sf S}_3(\gg)\cong (\ZZ/2)^{\oplus \infty}.$ By Corollary \ref{cor_H_3S} we obtain 
$$H_3^{\sf S}(\gg)\cong H^{\sf CE}_{2,1}(\gg).$$
Then it is enough to compute $H^{\sf CE}_{2,1}(\gg)$ using Theorem \ref{th_H21}. For any free abelian group $A$ there is a natural isomorphism ${\sf Tor}(A/2,A/2)\cong A/2\otimes A/2.$ It follows that $$ \Omega (A/2) \cong   \Gamma^2_{\ZZ/2}(A/2),\hspace{1cm} \gamma_2(a) \leftrightarrow \omega_2^2(a).$$ Note that there is a natural epimorphism 
$$\varepsilon: \Gamma^2_{\ZZ/2}(A/2) \epi A/2, \hspace{1cm} \gamma_2(a)\mapsto a$$
that vanishes $\gamma_1(a)\gamma_1(b)$ for any $a,b\in A/2.$
This follows that 
$\Omega \gg\cong\Gamma^2_{\ZZ/2}(\gg)$ and the composition of two maps 
$$\Omega \gg \otimes \gg \overset{\varphi}\longrightarrow \Omega \gg \overset{\varepsilon}\longrightarrow \gg.$$
vanishes. Therefore there is an epimorphism 
$H^{\sf CE}_{2,1}(\gg)\epi \gg$
induced by $\varepsilon.$ The Lie algebra $\gg$ is isomorphic to $(\ZZ/2)^{\oplus \infty}$ as an abelian group. Therefore, we have an epimorphism 
$$H^{\sf CE}_{2,1}(\gg)\epi (\ZZ/2)^{\oplus \infty}.$$
On the other hand $H^{\sf CE}_{2,1}(\gg)$ is a quotient of $\Omega \gg \cong \Gamma^2_{\ZZ/2}((\ZZ/2)^{\oplus \infty})\cong (\ZZ/2)^{\oplus \infty},$ and hence $H^{\sf CE}_{2,1}(\gg)\cong (\ZZ/2)^{\oplus \infty}.$ 

Let us compute $H_*^{\sf T}(\gg).$ It is easy to check that $U \gg=\ZZ\langle x,y \rangle/(2x,2y).$ Denote by $I$ the augmentation ideal of $U:=\ZZ\langle x,y \rangle/(2x,2y).$ Note that $I$ coincides with the augmentation ideal of $\mathbb F_2\langle x,y \rangle$ which is free as a right module
$$(\mathbb F_2\langle x,y \rangle)^2\cong I, \hspace{1cm} (a,b)\mapsto xa+yb.$$ 
Using the exact sequence 
$$ 0 \longrightarrow \ZZ \overset{2\cdot }\longrightarrow  U \longrightarrow  \mathbb F_2\langle x,y \rangle \longrightarrow 0$$
we obtain that there is a short exact sequence
$$0 \longrightarrow \ZZ^2 \longrightarrow  U^2 \longrightarrow  I \longrightarrow 0.$$
Combining this with the short exact sequence $0\to I \to U \to \ZZ\to 0,$ we obtain the following free resolution 
$$\dots \to U^8 \to  U^8 \overset{f^4}\to U^4 \overset{g^4}\to U^4 \overset{f^2}\to U^2 \overset{g^2}\to U^2 \overset{f}\to U \to 0,$$
where $f(a,b)=xa+yb$ and $g(a)=2\varepsilon(a).$ Tensoring it by $\ZZ$ over $U$ we obtain the following complex
$$ \dots \to \ZZ^8 \overset{2}\to  \ZZ^8 \overset{0}\to \ZZ^4 \overset{2}\to \ZZ^4 \overset{0}\to \ZZ^2 \overset{2}\to \ZZ^2 \overset{0}\to \ZZ \to 0.$$
It follows that $H_{2n+2}^{\sf T}(\gg)=0$ and $H_{2n+1}^{\sf T}(\gg)=(\ZZ/2)^{2^{n+1}}$ for $n\geq 0.$
\end{proof}

\section{Relative purity}

\subsection{Pure acyclic complexes}

In this section we assume that $R$ is a unital associative ring.

\begin{definition}
A short exact sequence of right $R$-modules
$$E: \hspace{1cm} 0 \longrightarrow M'\longrightarrow  M \longrightarrow M'' \longrightarrow 0 \hspace{1cm}\  $$
is called {\it pure exact} if one of the following equivalent properties hold (see \cite[Th. 4.89]{lam2012lectures}):
\begin{enumerate}
\item $E\otimes_R N$  is short exact  for any left $R$-module $N$;

\item  $E\otimes_R N$ is short exact  for any finitely presented left $R$-module $N$;

\item ${\sf Hom}_R(N,E)$ is short exact for any finitely presented right $R$-module $N;$

\item $E$ is a filtered colimit of split short exact sequences
$E_i.$
\end{enumerate}
A submodule $M' \leq M$ is called {\it pure} if the sequence $M'\mono M\epi M/M'$ is pure exact. 
\end{definition}

\begin{definition}
An acyclic complex $C_\bullet$ of $R$-modules is called {\it pure acyclic} if one of the following equivalent properties hold (see \cite[Prop.2.2]{emmanouil2016pure} and  \cite[Th.18-2.11]{dauns1994modules}): 
\begin{enumerate}
\item $Z_n\mono C_n \epi Z_{n+1}$ is pure exact for any $n,$ where $Z_k={\sf Ker}(C_k\to C_{k-1});$

\item $C_\bullet\otimes_R N$ is acyclic for  any left $R$-module $N$;

\item $C_\bullet\otimes_R N$ is acyclic for any finitely presented left  $R$-module $N;$
\item ${\sf Hom}_R(N,C_\bullet)$ is acyclic for any finitely presented right $R$-module $N$;
\item $C_\bullet$ is a filtered colimit of contractible complexes. 
\end{enumerate}
\end{definition}

\begin{lemma}[pure $3\times 3$-lemma]\label{lemma_3x3} Let $C'_\bullet \mono C_\bullet \epi C''_\bullet$ be a short exact sequence of complexes of right $R$-modules such that 
$C'_n\mono C_n \epi C''_n$ is pure exact sequence for any $n.$ Assume that any two of this three  complexes $C'_\bullet,C_\bullet,C''_\bullet$ are pure acyclic. Then the third one is also pure acyclic.  
\end{lemma}
\begin{proof}
Since $C'_n\mono C_n \epi C''_n$ is pure exact for any $n,$ then $C'_\bullet \otimes_R N \mono C_\bullet \otimes_R N \epi C''_\bullet\otimes_R N$ is a short exact sequence for any left $R$-module $N.$ The assertion follows from the long homology sequence.
\end{proof}

\subsection{Relative homology.} In this section we denote by $R$ a ring and by $S$ its subring. We will use the framework of  $(R,S)$-relative homological algebra. We follow notation and terminology of Mac Lane \cite[Ch.IX]{maclane2012homology}. 
 
Assume that $M$ is a right $R$-module and $N$ is a left $R$-module. Usually in order to define ${\sf Tor}_*^{(R,S)}(M,N)$ one takes an $S$-split $(R,S)$-projective resolution $P_\bullet\epi M$ and define
$${\sf Tor}_*^{(R,S)}(M,N):=H_*(P_\bullet \otimes_R N).$$ 
We generalize the notion of $S$-split resolution in the following way. 
We say that a resolution $P_\bullet \epi M$ of an $R$-module $M$ is {\it $S$-pure}, if the augmented resolution $P'_\bullet={\sf Cone}(P_\bullet \to M[0])$ is pure acyclic over $S$. The aim of this section is to prove that in order to compute ${\sf Tor}_*^{(R,S)}(M,N)$ one can use not only $S$-split $(R,S)$-projective resolutions but also $S$-pure $(R,S)$-projective resolutions. 

\begin{lemma}\label{lemma_tensor_S-pure} Let $C_\bullet$ be an $S$-pure acyclic complex of right $R$-modules and let $P$ be a left $(R,S)$-projective $R$-module. Then the complex $C_\bullet\otimes_R P$ is acyclic. 
\end{lemma}
\begin{proof}
Any $(R,S)$-projective left $R$-module is a direct summand of a module of the form $R\otimes_S A,$ where $A$ is a left $S$-module. Then it is enough to prove the statement for $P=R\otimes_S A.$ 
The assertion follows from the isomorphism $C_\bullet\otimes_R (R\otimes_S A )\cong C_\bullet \otimes_S A$ and the fact that $C_\bullet$ is $S$-pure acyclic.  
\end{proof}

\begin{lemma}\label{lemma_balance} Let $M$ and $N$ be right and left $R$-modules respectively and $P_\bullet\to M$ and $Q_\bullet \to N$ be their $S$-pure $(R,S)$-resolutions. Then the maps $P_\bullet\to M$ and $Q_\bullet\to N$ induce isomorphisms
$$H_*(P_\bullet\otimes_R N ) \cong H_*(P_\bullet \otimes_R Q_\bullet ) \cong H_*(M \otimes_R Q_\bullet).$$ 
\end{lemma}
\begin{proof}
Consider the double complex $D_{n,m}=P_n\otimes_RQ_m.$ Then $P_\bullet \otimes_R  Q_\bullet$ is the totalization of this complex. By Lemma \ref{lemma_tensor_S-pure} the complex $P'_\bullet \otimes_R Q_m$ is acyclic, where $P'_\bullet$ is the augmented resolution. It follows that $H_*(P_\bullet \otimes_R Q_m)=M\otimes_R Q_m [0]$ for any $m.$  Therefore the horizontal homology of $D_{\bullet\bullet}$ are concentrated in zero row and equal to $M\otimes Q_\bullet.$ Similarly the vertical homology of $D_{\bullet \bullet}$ are concentrated in zero column and equal to $P_\bullet \otimes_R N.$ Therefore the two spectral sequences of the double complex imply the assertion. 
\end{proof}

\begin{theorem}\label{theorem_pure_rel_gen}
Let $M$ be right $R$-module and  $P_\bullet\to M$ be its $S$-pure $(R,S)$-projective resolution. Then for any left $R$-module there is a natural isomorphism
$${\sf Tor}^{(R,S)}_*(M,N)\cong H_*(P_\bullet\otimes_R N).$$ 
\end{theorem}
\begin{proof}
Consider an $S$-split  $(R,S)$-projective resolution  $Q_\bullet\to N.$ Then ${\sf Tor}^{(R,S)}_\bullet(M,N)$ is isomorphic to $ H_*(M\otimes_R Q_\bullet).$ Since $P_\bullet$ is $S$-split $(R,S)$-projective resolution, it is also $S$-pure $(R,S)$-projective resolution. Then 
the assertion follows from Lemma \ref{lemma_balance}.
\end{proof}

The following statement seems to be well known but we add it for completeness.

\begin{proposition} \label{prop_tor_rel-ord}
Assume that $R$ is flat as a right $S$-module and $M$ is a right $R$-module which is flat over $S$. Then 
$${\sf Tor}^{(R,S)}_*(M,N)\cong {\sf Tor}^{R}_*(M,N).$$
\end{proposition}
\begin{proof} Note that for any flat right $S$-module $A$ the $R$-module $A\otimes_S R$ is also flat. On the other hand, since $R$ is flat over $S,$ any flat right $R$-module is flat over $S$ (indeed any flat $R$-module is a filtered colimit of free $R$-modules which are flat over $S$). It follows that, if $A$ is a flat $S$-module, then $A\otimes_S R$ is also a flat $S$-module. Therefore  $M\otimes_SR^{\otimes_S n}$ is a flat $R$-module for any $n\geq 1$.

Following the book of Mac Lane we denote by $\beta(R)$ the bar-complex, whose components are $\beta_n(R)=R^{\otimes_S (n+2)}$ \cite[Ch.IX, Th.8.1]{maclane2012homology}.  Then ${\sf Tor}^{(R,S)}_*(M,N)=H_*(M\otimes_R \beta(R)\otimes_R N).$ On the other hand $M\otimes_R \beta(R) $ is a flat resolution of $M$ because its components $M\otimes_R R^{\otimes_S (n+2)} \cong M\otimes_S R^{\otimes_S (n+1)}$ are flat $R$-modules and it is a resolution of $M$  (see \cite[Ch.IX, Cor.8.2]{maclane2012homology}). Therefore ${\sf Tor}^R_*(M,N)\cong (M\otimes_R \beta(R))\otimes_R N\cong {\sf Tor}^{(R,S)}_*(M,N).$
\end{proof}

\section{Koszul complexes}\label{section_Koszul}

\subsection{Preliminaries}
Let $\kk$ be a commutative ring and $\mathcal A=\bigoplus_{n\in  \ZZ} \mathcal A_n $ be a graded algebra over $\kk$. A {\it differential} on $\mathcal A$ is a collection of $\kk$-homomorphisms  $(\partial : \mathcal A_n\to\mathcal A_{n-1})_{n\in \ZZ}$ satisfying $\partial \partial=0$ and $\partial(\alpha\beta)=\partial (\alpha)\beta + (-1)^n\alpha\partial (\beta) $ for any $\alpha\in\mathcal  A_n$ and $\beta\in \mathcal A_m.$ Then a \emph{dg-algebra} is a graded algebra together with a differential.

If $\kk$ is a commutative ring, $E$ is an $\kk$-module and $f:E\to \kk$ is a $\kk$-homomorphism, then the associated Koszul chain complex ${\sf Kos}_\bullet(f)$ is a dg-algebra whose underlying graded algebra is  the exterior algebra $\Lambda_\kk E$ and the differential is given by the formula
$$ \partial(e_1\wedge \dots \wedge e_n)=\sum_{i=1}^n (-1)^{i+1} f(e_i)e_1\wedge \dots \wedge \hat e_i \wedge \dots \wedge e_n.$$
It is easy to check that this formula defines a differential.

\begin{lemma}\label{lemma_Kos(f)} The differential of ${\sf Kos}_\bullet(f)$ is the unique differential on the graded algebra $\Lambda_\kk E$  such that $\partial(e)=f(e)$ for any $e\in E.$
\end{lemma}
\begin{proof}
Obvious.
\end{proof}

\subsection{Koszul complex of a module}
 In this subsection we assume that $\kk$ is a fixed commutative ring and all modules are modules over $\kk$ and set $\otimes=\otimes_\kk.$

Let $M$ be a $\kk$-module. Consider the symmetric algebra of this module $SM$,  an $SM$-module $E:=SM\otimes M$ and $f_M:SM\otimes M\to SM$  given by multiplication $f_M(\alpha\otimes a)=\alpha a.$ Since the extension of scalars sends exterior powers to exterior powers, we have  an isomorphism
$$ \Lambda_{SM}^n (SM\otimes M) \cong SM \otimes \Lambda^nM. $$
Therefore we obtain a dg-algebra  ${\sf Kos}_\bullet(M)\cong{\sf Kos}_\bullet(f_M)$ over $SM$ such that  
$${\sf Kos}_n(M)=SM\otimes \Lambda^nM$$
with the differential
$$\partial(\alpha\otimes a_1\wedge \dots \wedge a_n )=\sum_{i=1}^n (-1)^{i+1} \alpha a_i \otimes a_1\wedge \dots \wedge  \hat a_i \wedge \dots \wedge a_n.$$

$${\sf Kos}_\bullet(M):  \hspace{1cm} \dots \to  SM\otimes \Lambda^2 M\to  SM\otimes M \to SM\otimes \kk \to 0$$
(see {\cite[4.3.1]{illusie2006complexe}}).

\begin{lemma}\label{lemma_Kos(A)} Consider the algebra $SM\otimes \Lambda M$ as a graded algebra with the grading given by  $$(SM\otimes \Lambda M)_n=SM\otimes \Lambda^n M.$$ Then the differential of ${\sf Kos}_\bullet(M)$ is the unique differential  on the graded algebra $SM\otimes \Lambda M$ such that $\partial(1\otimes a)=a\otimes 1$ for any $a\in M.$
\end{lemma}
\begin{proof}
Since $(SM\otimes \Lambda M)_0=SM\otimes \kk$ we obtain $\partial (\alpha \otimes 1 )=0$ for any $\alpha\in SM$ and any differential $\partial.$
Therefore, the equation $\partial(1\otimes a)=a\otimes 1$ implies $\partial(\alpha\otimes a) =\alpha a\otimes 1=f_A(\alpha\otimes  a)\otimes 1.$  Then the assertion follows from Lemma \ref{lemma_Kos(f)}. 
\end{proof}

\begin{lemma}\label{lemma_direct_summ_Kos} For any $\kk$-modules $M,N$ there is a natural isomorphism of dg-algebras
$${\sf Kos}_\bullet(M\oplus N)\cong {\sf Kos}_\bullet(M)\otimes {\sf Kos}_\bullet(N).$$
\end{lemma}
\begin{proof}
Since we have isomorphisms of graded rings $SM\otimes SN \cong S(M\oplus N)$ and $ \Lambda M\otimes \Lambda N\cong \Lambda(M\oplus N)$ we obtain an isomorphism of graded rings $S(M\oplus N) \otimes \Lambda(M\oplus N)\cong (SM\otimes \Lambda M)\otimes (SN \otimes \Lambda N).$ Then we only need to check that this isomorphism is compatible with differentials. On the other hand for any $a\in A,b\in B$ the element $1\otimes (a,b)\in {\sf Kos}_\bullet(M\oplus N)$ corresponds to the element $(1\otimes a)\otimes 1 +1\otimes (1\otimes b)$ of ${\sf Kos}_\bullet(M)\otimes {\sf Kos}_\bullet(N).$ Then $\partial((1\otimes a)\otimes 1 +1\otimes (1\otimes b))=a\otimes 1+ 1\otimes b$ which corresponds to $(a,b)\otimes 1$ in ${\sf Kos}_\bullet(M\oplus N).$ Therefore the assertion follows from Lemma \ref{lemma_Kos(A)}.
\end{proof}

\begin{lemma}\label{lemma_Kos_direct_lim} The functor from the category of modules to the category of complexes  
$${\sf Kos}_\bullet: {\sf Mod}\longrightarrow {\sf Com}, \hspace{1cm} M\mapsto {\sf Kos}_\bullet(M)$$
commutes with filtered colimits.
\end{lemma}
\begin{proof}
It follows from the fact that the functors of symmetric powers, exterior powers and tensor product commute with filtered colimits and the differential is natural in $M$.
\end{proof}

\begin{lemma}\label{lemma_Kos_contractible}
Assume that $M$ is a $\kk$-module, which is isomorphic to a finite direct sum of cyclic $\kk$-modules. Then the augmentation of ${\sf Kos}_\bullet(M)$ is a homotopy equivalence of complexes  $${\sf Kos}_\bullet(M)\overset{\simeq}\to \kk[0].$$
(The homotopy is not natural in $M$.) 
\end{lemma}
\begin{proof}  The proof is by induction on the number of summands.
If $M\cong \kk/I$ is cyclic, then $\Lambda^nM=0$ for $n\geq 2,$ and $S^nM\cong \kk/I\cong S^nM\otimes M$ and $n\geq 1.$ Then $SM=\kk\oplus \left(\bigoplus_{n\geq 1} \kk/I\right)$ and $SM\otimes M= \bigoplus_{n\geq 1} \kk/I.$ It follows that $ SM\otimes M \to SM\to \kk$ is a split short exact sequence. 

Assume that $M,N$ are $\kk$-modules such that ${\sf Kos}_\bullet(M)$  and ${\sf Kos}_\bullet(N)$ are homotopy equivalent to $\kk[0].$  By Lemma  \ref{lemma_direct_summ_Kos} we have ${\sf Kos}_\bullet(M\oplus N)\cong {\sf Kos}_\bullet(M)\otimes {\sf Kos}_\bullet(N).$ Homotopy equivalence is invariant with respect to the tensor product. The assertion follows.
\end{proof}

We denote by ${\sf Kos}'_\bullet(M)$ the  ``augmented'' Koszul complex, where ${\sf Kos}'_{-1}(M)=\kk$ and ${\sf Kos}'_{0}(M)\to {\sf Kos}'_{-1}(M)$ is the projection $SM\to S^0M\cong \kk.$

\begin{proposition}\label{prop_kos} Let  $M$ be a  $\kk$-module. Assume that 
$M$ can be presented as a filtered colimit of finite direct sums of cyclic modules
$$M={\sf colim}\:  M_i, \hspace{1cm} M_i\cong \kk/I_{1,i} \oplus \dots \oplus \kk/I_{n_i,i}.$$ Then the complex ${\sf Kos}'_\bullet(M)$ is pure acyclic. 
\end{proposition}
\begin{proof}
Since $M$ is a direct limit of finite direct sums of cyclic modules, the complex ${\sf Kos}'_\bullet(M)$ is a direct limit of contractible complexes by Lemma \ref{lemma_Kos_contractible} and Lemma \ref{lemma_Kos_direct_lim}. Hence, it is pure acyclic. 
\end{proof}

The Koszul complex ${\sf Kos}_\bullet(M)$ has a natural decomposition into direct sum of complexes
$${\sf Kos}_\bullet(M)=\bigoplus_{k\geq 0} {\sf Kos}_\bullet(M,k),$$
where ${\sf Kos}_n(M,k)= S^{k-n}M \otimes \Lambda^n M:$
$$ {\sf Kos}_\bullet(M,k): \hspace{1cm} 0 \to \Lambda^k M \to M\otimes \Lambda^{k-1}M \to  \dots  \to  S^{k-1}M\otimes M \to  S^kM\to 0.
$$

\begin{theorem}\label{th_kos} 
Let $M$ be a module over a commutative ring $\kk$ and $k\geq 1$ be a natural number. Assume that one of the following properties hold: 
\begin{enumerate}
    \item $\kk$ is a principal ideal domain;
    \item $M$ is flat;
    \item  $k \cdot 1_\kk$ is invertible in $\kk.$
\end{enumerate}
Then the complex ${\sf Kos}_\bullet(M,k)$ is pure acyclic. Moreover, in the case (3) it is contractible. 
\end{theorem}
\begin{proof}
(1) If $\kk$ is a principal ideal domain, any finitely generated module is a finite direct sum of cyclic modules. Since any module can be presented as a direct limit of its finitely generated submodules, any module is a filtered colimit of finite sums of cyclic modules. Then it follows from Proposition \ref{prop_kos} and the fact that ${\sf Kos}_\bullet(M,k)$ is a direct summand of ${\sf Kos}'_\bullet(M,k)$ for $k\geq 1.$

(2) If $M$ is flat then by the Govorov-Lazard theorem it is a filtered colimit of finitely generated free modules. Then again by Proposition \ref{prop_kos} we obtain the assertion. 

(3) Here we prove that the complex ${\sf Kos}(M,k)$ is contractible if $k\cdot {\sf 1}_\kk$ is invertible in $\kk$. Consider the map 
$$h_{n,m}:S^nM \otimes \Lambda^m M \longrightarrow S^{n-1}M\otimes \Lambda^{m+1}M $$
given by
$$h_{n,m}(a_1\dots a_n \otimes \beta) = \sum_{i=1}^n a_1\dots \hat a_i \dots a_n \otimes a_i \wedge \beta.$$ We also consider the differentials
$\partial_{n,m}:S^nM \otimes \Lambda^mM\to S^{n+1}M \otimes \Lambda^{m-1}M.$
We claim that 
\begin{equation}\label{eq_homot_Kos}
\partial_{n-1,m+1} h_{n,m} + h_{n+1,m-1} \partial_{n,m} = (n+m) \cdot {\sf id}.
\end{equation}

Prove the equation \eqref{eq_homot_Kos}.  Consider elements  $\alpha=a_1 \dots  a_n\in S^n M$ and $\beta=b_1 \wedge \dots \wedge b_m\in \Lambda^m M.$ We set $\alpha_i=a_1 \dots \hat a_i \dots a_n$ and $\beta_j=b_1 \wedge \dots \wedge \hat b_j \wedge \dots \wedge b_m.$ Then $ h(\alpha\otimes \beta )=\sum_{i=1}^n \alpha_i \otimes a_i \wedge \beta$ and $ \partial(\alpha \otimes \beta )=\sum_{j=1}^m (-1)^{j+1} \alpha b_j\otimes \beta_j.$ 

\begin{equation*}
\begin{split}
\partial_{n-1,m+1} h_{n,m} (\alpha \otimes \beta ) &=  \partial_{n-1,m+1} \left(  \sum_{i=1}^n \alpha_i \otimes a_i \wedge \beta\right)\\
&=  \sum_{i=1}^n \left( \alpha \otimes \beta +  \sum_{j=1}^m (-1)^{j} \alpha_ib_j \otimes a_i \wedge \beta_j \right)\\
&=   n \alpha \otimes \beta +  \sum_{i=1}^n \sum_{j=1}^m (-1)^{j} \alpha_ib_j \otimes a_i \wedge \beta_j.
\end{split}
\end{equation*}

\begin{equation*}
\begin{split}
h_{n+1,m-1}\partial_{n,m}( \alpha \otimes \beta ) &= h\left( \sum_{j=1}^m (-1)^{j+1} \alpha b_j \otimes \beta_j \right) \\
& = \sum_{j=1}^m ( (-1)^{j+1} \alpha \otimes b_j \wedge \beta_j +   \sum_{i=1}^n (-1)^{j+1} \alpha_i b_j \otimes a_i \wedge \beta_j )\\
& = \sum_{j=1}^m (  \alpha \otimes  \beta +   \sum_{i=1}^n (-1)^{j+1} \alpha_i b_j \otimes a_i \wedge \beta_j )
\\
& = m \alpha\otimes \beta + \sum_{i=1}^n \sum_{j=1}^m (-1)^{j+1} \alpha_i b_j \otimes a_i \wedge \beta_j,
\end{split}    
\end{equation*}
 Therefore
$\partial_{n-1,m+1} h_{n,m} + h_{n+1,m-1} \partial_{n,m} = (n+m) \cdot {\sf id}.$
It follows that $\tilde h_i = \frac{1}{k} \cdot h_{k-i,i}$ is a contracting homotopy for ${\sf Kos}(M,k).$
\end{proof}

\subsection{Example}

This section is devoted to an example of a commutative ring $\kk$ and a finitely presented module $M$ such that ${\sf Kos}_\bullet(M,n)$ is not exact. Namely, for each $n$ we construct $\kk$ and $M$ such that the map $\partial:\Lambda^n M\to M \otimes \Lambda^{n-1}M $ given by 
$$\partial(a_1\wedge \dots \wedge a_n)= \sum_{i=1}^n (-1)^{i+1} a_i \otimes a_1\wedge \dots \wedge \hat a_i \wedge \dots \wedge a_n$$ is not a monomorphism. 

The example that we construct is motivated by an example of a Lie algebra $\gg$
 over a commutative ring $\kk$ that does not satisfy  the Poincar\'e-Birkhoff-Witt property given in \cite[\S 5]{cohn1963remark}.

\begin{proposition}\label{prop_counter_Kos} Let $n$ be a natural number and $p$ be a prime dividing $n.$ Assume that $\kk=\mathbb F_p[t_1,\dots,t_{n+1}]/(t_1^2,\dots,t_{n+1}^2)$ and $M$ is a $\kk$-module generated by $e_1,\dots,e_{n+1}$ modulo relation $t_1e_1+\dots+t_{n+1}e_{n+1}=0.$ 
$$M=\kk^{n+1}/\langle  t_1e_1+\dots+t_{n+1}e_{n+1} \rangle$$
Then the map
$\partial: \Lambda^nM\to M\otimes \Lambda^{n-1} M$
is not injective, and hence the complex ${\sf Kos}_\bullet(M,n)$ is not acyclic. More precisely, the element $t_1e_1 \wedge t_2e_2 \wedge \dots \wedge t_ne_n\in \Lambda^n M$ is a non-trivial element of  ${\sf Ker}(\partial: \Lambda^nM\to M\otimes \Lambda^{n-1} M).$  
\end{proposition}
\begin{proof} Set $x_i=t_ie_i.$ We need to prove two statements:
(A) $\partial(x_1\wedge \dots \wedge x_n)=0;$
(B) $x_1\wedge \dots \wedge x_n\ne 0.$

(A) Note that $x_i\otimes x_i=t_i^2(e_i\otimes e_i)=0$ in $M\otimes M.$ Moreover, by the same reason for any $f\in M^{\otimes k}$ we have $x_i\otimes f \otimes x_i=0$ in $M^{\otimes k+2}.$
We claim that for any permutation from the symmetric group $\sigma\in \Sigma_{n+1}$ we have the following equation in $M^{\otimes n}:$
\begin{equation}\label{eq_perm}
 x_{\sigma(1)} \otimes x_{\sigma(2)} \otimes  \dots \otimes x_{\sigma(n)}={\sf sign}(\sigma)\cdot  x_{1} \otimes x_{2} \otimes \dots \otimes x_{n}.   
\end{equation}
Since permutations of the form $(i,n+1)$ generate the symmetric group $\Sigma_{n+1},$ it is enough to check that, if $\sigma =\tau \circ (i,n+1),$ then 
$$x_{\sigma(1)} \otimes x_{\sigma(2)} \otimes  \dots \otimes x_{\sigma(n)}= -  x_{\tau(1)} \otimes x_{\tau(2)} \otimes \dots \otimes x_{\tau(n)}.$$ 
Set $f=x_{\sigma(1)}\otimes \dots x_{\sigma(i-1)}$ and $g=x_{\sigma(i+1)}\otimes \dots x_{\sigma(n)}.$ Then we need to prove that 
$$f\otimes x_{\sigma(i)} \otimes g=- f\otimes x_{\sigma(n+1)} \otimes g.$$
Note that $f\otimes x_{\sigma(j)}\otimes g=0$ for $j\notin \{i,n+1\},$ because $x_{\sigma(j)}$ arises in $f$ or $g.$ Using this we obtain 
$$f\otimes x_{\sigma(i)}\otimes g=f\otimes \left( -\sum_{j\ne i} x_{\sigma(j)}  \right)\otimes g= -f\otimes x_{\sigma(n+1)}\otimes g.$$
So we have proved the equation \eqref{eq_perm}.

The equation \eqref{eq_perm} implies a similar equation in $M\otimes \Lambda^{n-1}M$ 
$$x_{\sigma(1)} \otimes x_{\sigma(2)} \wedge  \dots \wedge x_{\sigma(n)}={\sf sign}(\sigma)\cdot  x_{1} \otimes x_{2} \wedge \dots \wedge x_{n}.$$
In particular 
$$(-1)^{i+1} x_i \otimes x_1\wedge \dots \wedge \hat x_i \wedge \dots \wedge x_n= x_1 \otimes x_2\wedge  \dots \wedge x_n.$$
It follows that 
$$\partial(x_1 \wedge \dots \wedge x_n)=n \cdot x_1 \otimes x_2 \wedge \dots \wedge x_n=0$$
because the characteristic $p$ divides $n.$ 

(B) Prove that $x_1 \wedge \dots \wedge x_n\ne 0.$ In this proof we use
the theory of Gr\"obner bases for modules over polynomial rings that can be found in \cite[Ch.3]{adams1994introduction}.  The module $M$ can be considered as a module over the polynomial ring $\mathbb F_p[t_1,\dots,t_{n+1}].$ Denote by $F=\mathbb F_p[t_1,\dots,t_{n+1}]^{n+1}$ the free module over the polynomial ring generated by $e_1,\dots,e_{n+1}.$ Then the kernel of the map $F\epi M$ is generated by the elements of the form $t_i^2 e_j $ and the element $t_1e_1+\dots+t_{n+1}e_{n+1}.$ Therefore the kernel  $K={\sf Ker}(\Lambda^nF\to \Lambda^n M)$ is generated by the elements of two types
\begin{enumerate}
\item[(1)] $t_{i}^2 e_{j_1}\wedge \dots \wedge e_{j_n},$ where $j_1<\dots <j_n;$
\item[(2)] $(t_1e_1+\dots+t_{n+1}e_{n+1}) \wedge e_{j_2} \wedge \dots \wedge e_{j_n},$ where $j_2<\dots<j_n.$
\end{enumerate}

For $i\in \{1,\dots,n\}$ we set $E_i=e_1 \wedge \dots \wedge \hat e_i \wedge \dots \wedge e_{n+1}\in \Lambda^n F.$ Then $E_1,\dots,E_{n+1}$ is a basis of $\Lambda^n F.$ Similarly we set $E_{i,j}=e_1 \wedge \dots \wedge \hat e_i \wedge \dots \wedge \hat e_j \wedge \dots \wedge e_{n+1} \in  \Lambda^{n-1} M$ for $i<j.$ Then the generators of $K$ can be rewritten as follows
(1) $t_i^2 E_j;$ 
(2) $(t_1e_1+\dots+t_{n+1}e_{n+1}) \wedge E_{i,j}.$ Note that 
$$e_i\wedge E_{i,j}=(-1)^{i+1} E_j, \hspace{1cm} e_j \wedge E_{i,j}=(-1)^jE_i $$
and $e_k \wedge E_{i,j}=0$ for $k\notin \{i,j\}.$ Using these identities we obtain that the generators of $K$ can be presented as
\begin{itemize}
\item[(G1)] $t_i^2E_j$
\item[(G2)] $t_jE_i+ (-1)^{i+j+1}t_iE_j$ for $i<j.$
\end{itemize}

It is easy to see that the following elements also lie in $K$
\begin{itemize}
\item[(G3)] $t_it_jE_i$ for all $i,j.$
\end{itemize}
We consider the following order on generators of the ring $t_1<\dots<t_{n+1}$ and the opposite order on generators of the module $E_{n+1}<\dots<E_1.$ Then  the leading term of $t_jE_i+ (-1)^{i+j+1}t_iE_j$ for $i<j$ is $t_jE_i.$
Using \cite[Theorem 3.5.19]{adams1994introduction}, it is easy to check that (G1),(G2),(G3) form a Gr\"obner basis of $K.$   Then the reduced form of the element $x_1\wedge \dots \wedge x_n$ with respect to the Gr\"obner basis is $t_1\dots t_n E_{n+1}.$ Since the reduced form is non-trivial, the element is non-trivial. 
\end{proof}

\begin{corollary}\label{cor_counter_kos}
Let $\kk=\mathbb F_2[t_1,t_2,t_3]/(t_1^2,t_2^2,t_3^2)$ and $M=\kk^3/\langle t_1e_1+t_2e_2+t_3e_3 \rangle.$
Then the map
$$\partial: \Lambda^2M\to M\otimes  M, \hspace{1cm} \partial(a\wedge b)=a\otimes b - b\otimes a$$
is not injective. More precisely, the element $t_1e_1\wedge t_2e_2$ is a non-trivial element of ${\sf Ker}(\Lambda^2M\to M\otimes  M).$
\end{corollary}

\section{Relative Tor homology}

Let $\gg$ be a Lie algebra over a commutative ring $\kk$ and $M$ be a right $U\gg$-module. Then the  Chevalley-Eilenberg chain complex ${\sf CE}_\bullet(\gg,M)$ is a complex whose components are
$${\sf CE}_n(\gg,M)=M\otimes \Lambda^n \gg$$
and the differential is given by 
\begin{equation*}
\begin{split}
\partial(m\otimes x_1\wedge \dots \wedge x_n)=& \sum_{i=1}^n(-1)^{i+1}mx_i\otimes x_1\wedge \dots \wedge \hat x_i \wedge \dots \wedge x_n + \\
&  \sum_{i<j}(-1)^{i+j} m \otimes  [x_i,x_j] \wedge x_1\wedge \dots \wedge \hat x_i \wedge \dots \wedge \hat x_j \wedge \dots \wedge x_n.
\end{split}
\end{equation*}
It is easy to see that ${\sf CE}_\bullet(\gg)\cong {\sf CE}_\bullet(\gg,\kk).$

We define Chevalley-Eilenberg homology of $\gg$ with coefficients in $M$ as the homology of this complex
$$H^{\sf CE}_*(\gg,M):=H_*({\sf CE}_\bullet(\gg,M)).$$
We also define relative Tor  homology of $\gg$ as follows
$$H_*^{\sf RT}(\gg,M)={\sf Tor}^{(U\gg,\kk)}_*(M,\kk).$$

\begin{theorem}\label{theorem_CE_rel} Let $\kk$ be a commutative ring and $\gg$ be a  Lie algebra over $\kk.$ Assume that 
\begin{enumerate}
    \item either $\kk$ is a principal ideal domain, 
    \item or $\gg$ is flat over $\kk.$ 
\end{enumerate}
Then for any $U\gg$-module $M$ there is a natural isomorphism 
$$H^{\sf CE}_*(\gg,M)\cong  H^{\sf RT}_*(\gg,M).$$
\end{theorem}

In order to prove this theorem we need several lemmas. For simplicity we set $U=U\gg$ and denote by $F_nU=F_nU\gg$ the standard filtration on $U.$ We also set $F_\infty U=U$ and $F_{-1}U=0.$ We also denote by 
$$GU = \bigoplus_{n\geq 0} G_nU, \hspace{1cm} 
G_nU=F_nU/F_{n-1}U$$
the assotiated graded ring and consider the Poincar\'e–Birkhoff–Witt homomorphism
$$S\gg \to GU.$$

\begin{lemma}\label{lemma_commuting_filt_colim}
The functors of symmetric and exterior powers $S^n,\Lambda^n:{\sf Mod}\to {\sf Mod}$ and their Dold-Puppe derived functors $L_i S^n,L_i\Lambda^n :{\sf Mod}\to {\sf Mod}$ commute with filtered colimits.
\end{lemma}
\begin{proof}
Since the functors $S:{\sf Mod} \to {\sf CommAlg}$ and $\Lambda:{\sf Mod}\to {\sf GradCommAlg}$ are right adjoint, they commute with all colimits. It is well-known that filtered colimits of modules with additional algebraic structures commutes with forgetful functors to modules ${\sf CommAlg}\to {\sf Mod}$ and ${\sf GradCommAlg}\to {\sf Mod}$ (cf. \cite[Ch. IX, \S 1]{mclane1971categories}). Therefore $S^n, \Lambda^n:{\sf Mod}\to {\sf Mod}$  commute with filtered colimits. 

 For any module $M$ we consider its simplicial bar resolution $B_\bullet(M),$ where $B_0(M)=\kk^{\oplus M}$ and $B_{i+1}(M)=\kk^{\oplus B_i(M)}.$ It is the bar construction corresponding to the comonad $\GG$ on ${\sf Mod}$ given by $\GG(M)=\kk^{\oplus M}$ that comes from the free-forgetful adjunction ${\sf Sets}\leftrightarrows {\sf Mod}.$  It is easy to see that $\GG$ commutes with filtered colimits. Therefore the bar resolution $B_\bullet(M)$ also commutes with filtered colimits. Homology of a chain complex also commutes with filtered colimits, because filtered colimit is an exact functor. Therefore, $L_i S^n(-)= \pi_i(S^n (B_\bullet(-)))$ also commutes with filtered colimits. Similarly $L_i\Lambda^n$ commutes with filtered colimits.
\end{proof}

\begin{lemma}\label{lemma_flatU}
If $\gg$ is flat, then $U$ is flat. Moreover $F_nU/F_mU$ is flat for any $-1 \leq  m\leq n\leq \infty.$  
\end{lemma}
\begin{proof} Set $F_n=F_nU.$
Flat Lie  algebras satisfy PBW-property (see \cite[Th. 5.9]{grinberg2011poincare}, \cite{higgins1969baer}). Therefore we have a short exact sequence 
\begin{equation}\label{eq_ffs}
0\longrightarrow F_{n}/F_m \longrightarrow F_{n+1}/F_m \longrightarrow S^{n+1}\gg \longrightarrow 0    
\end{equation}
for any $-1 \leq m\leq n<\infty,$
where $S^n\gg$ is the symmetric power of $\gg.$ 

The functor $S^n$ sends free modules to free modules. Since it commutes with filtered colimits, and any flat module is a filtered colimit of free modules, we obtain that $S^n$ sends flat modules to flat modules. 
Therefore $S^n\gg$ is flat. 
Using the short exact sequence 
\eqref{eq_ffs} the fact that  $S^n\gg$ is flat and the fact  that an extension of flat modules is flat (\cite[Cor. 4.86]{lam2012lectures}), we obtain by induction that $F_n/F_m$ is flat for $n<\infty$. For $n=\infty$ we obtain this using the formula  $F_\infty/F_m=\varinjlim F_n/F_m.$ 
\end{proof}

\begin{lemma}\label{lemma_F_nU_pure} Under assumptions of Theorem \ref{theorem_CE_rel} $F_nU$ is a pure submodule of $U.$
\end{lemma} 
\begin{proof} 
(1) Assume that $\kk$ is a principal ideal domain. The proof is by induction on $n$. 
Prove the base case $n=0.$ Since $U$ is an augmented algebra,  $F_0U=\kk$ is a direct summand of $U,$ and hence,  it is a pure subgroup of $U.$  Assume that $F_nU$ is pure in $U$ and prove that $F_{n+1}U$ is pure in $U.$ By \cite[Lemma 26.1]{fuchs1970infinite}, it is enough to prove that $F_{n+1}U/F_nU$ is pure in $U/F_nU.$ Since $\kk$ is a principal ideal domain, it is enough to prove that the map $(F_{n+1}U/F_nU)\otimes \kk/a\to (U/F_nU)\otimes \kk/a$ is injective.

Let $a\in \kk.$  For a Lie algebra $\mathfrak{h}$ over $\kk/a$ we denote by $U^{\kk/a}(\mathfrak{h})$ the universal enveloping algebra over $\kk/a.$ By the universal property of the enveloping algebra we obtain $U\gg\otimes \kk/a\cong U^{\kk/a}(\gg\otimes \kk/a).$ Set $U^{\kk/a}=U^{\kk/a}(\gg\otimes \kk/a),$  $F_nU^{\kk/a}=F_nU^{\kk/a}(\gg\otimes \kk/a)$ and $GU^{\kk/a}:=GU^{\kk/a}(\gg\otimes \kk/a).$ We also denote by $S_{\kk/a}(-)$ the symmetric algebra over $\kk/a.$ Consider the following commutative diagram
$$
\begin{tikzcd}
S\gg\otimes \kk/a\arrow[r] \arrow[d] & GU\otimes \kk/a \arrow[d] \\
S_{\kk/a}(\gg\otimes \kk/a)\arrow[r] & GU^{\kk/a}.
\end{tikzcd}
$$
The left hand vertical map is obviously an isomorphism. The horizontal maps are isomorphisms by the Poincar\'e-Birkhoff-Witt theorem over $\kk$ and over $\kk/a$. Therefore the map $GU\otimes \kk/a\to GU^{\kk/a} $ is an isomorphism of graded algebras. 
By induction this implies that for any $k>n$ the map  $(F_{k}U/F_nU)\otimes \kk/a \to F_kU^{\kk/a}/F_nU^{\kk/a} $ is an isomorphism. Therefore the map $(U/F_nU)\otimes \kk/a \to U^{\kk/a}/F_nU^{\kk/a} $ is an isomorphism.  It follows that the map $(F_{n+1}U/F_nU)\otimes \kk/a\to (U/F_nU)\otimes \kk/a$  is injective because it is isomorphic to the map $F_{n+1}U^{\kk/a}/F_nU^{\kk/a} \to U^{\kk/a}/F_nU^{\kk/a}$. Hence $F_{n+1}U/F_nU$ is pure in $U/F_nU.$

(2) Assume that $\gg$ is flat. Then by Lemma \ref{lemma_flatU} we obtain that $U/F_nU$ is flat. It follows that the sequence $0\to F_nU \to U \to U/F_nU\to 0$ is pure exact. 
\end{proof}

Further we consider the following complex $$V_\bullet:={\sf CE}_\bullet(\gg, U\gg)$$
that we call the Chevalley-Eilenberg resolution. 
Since $F_kU$ is a pure subgroup of $U,$ the map $F_kU\otimes \Lambda^n \gg\to U \otimes \Lambda^n \gg$ is injective. Hence we can identify $F_kU\otimes \Lambda^n \gg$ with its image in $U\otimes \Lambda^n \gg:$
$F_kU\otimes \Lambda^n \gg \subseteq U\otimes \Lambda^n\gg.
$
We set $$F_kV_n:=F_{k-n}U\otimes \Lambda^n\gg \subseteq U\otimes \Lambda^n \gg.$$

\begin{lemma} Under assumptions of Theorem \ref{theorem_CE_rel} 
the submodules $F_kV_n$ define a filtration $F_kV_\bullet$ by subcomplexes on the complex $V_\bullet.$ 
\end{lemma}
\begin{proof}
Note that if
$u\otimes x_1\wedge \dots \wedge x_n\in F_kU\otimes \Lambda^n\gg,$ then $ux_i\otimes x_1\wedge \dots \wedge \hat x_i \wedge \dots \wedge x_n$ and $u\otimes [x_i,x_j]\wedge x_1\wedge \dots \wedge \hat x_i \wedge \dots \wedge \hat x_j \wedge \dots \wedge x_n $ are in $F_{k+1}U\otimes \Lambda^{n-1}.$ Therefore $\partial( F_kU\otimes \Lambda^n \gg )\subseteq F_{k+1}U \otimes \Lambda^{n-1}\gg.$ 
\end{proof}

\begin{lemma}\label{lemma_F_kV_pure} Under assumptions of Theorem \ref{theorem_CE_rel} the group $F_kV_n$ is pure in $V_n.$
\end{lemma}
\begin{proof}By Lemma \ref{lemma_F_nU_pure} the group $F_{k-n}U$ is pure in $U$. Hence $F_kV_n$ is pure in $V_n.$
\end{proof}

The adjoint graded complex of the filtration $F_kV_\bullet$ is denoted by 
$$GV_\bullet:=\oplus_{k\geq 0}\: G_kV_\bullet,\hspace{1cm} G_kV_\bullet := F_kV_\bullet/F_{k-1}V_\bullet.$$

\begin{lemma}\label{lemma_kos_GV_iso} Under assumptions of Theorem \ref{theorem_CE_rel} the PBW-isomorphism $S \gg \cong GU$ induces an isomorphism of graded complexes 
$${\sf Kos}_\bullet(\gg)\cong GV_\bullet.
$$ 
\end{lemma}
\begin{proof} First we remind that under assumption of Theorem \ref{theorem_CE_rel} the PBW-property is satisfied (see \cite[Th. 5.9]{grinberg2011poincare}, \cite{higgins1969baer}) and we have an isomorphism $S\gg\cong GU$.  By definition of the filtration $F_kV_\bullet$ we have 
$$F_kV_n/F_{k-1}V_n=(F_{k-n}U\otimes \Lambda^n\gg)/(F_{k-1-n} U \otimes \Lambda^n\gg)\cong $$
$$\cong (F_{k-n}U/F_{k-n-1}U) \otimes \Lambda^n\gg\cong S^{k-n}\gg \otimes \Lambda^n\gg.$$ 
Hence the isomorphism $S \gg \cong GU$ induces an isomorphism of components of these complexes. We need to check that this isomorphism is compatible with the differentials. The differential of $V_\bullet$ consists of two summands. Note that if $u\otimes x_1\wedge \dots \wedge x_n \in F_kV_n,$ then $u\otimes [x_i,x_j]\wedge x_1\wedge \dots \wedge \hat x_i \wedge \dots \wedge \hat x_j \wedge \dots \wedge x_n\in F_{k-1} V_{n-1}.$ It follows that in the quotient complex $G_kV_\bullet=F_kV_\bullet/F_{k-1}V_\bullet$ the second summand vanishes. The first summand of the differential of the complex $V_\bullet$ obviously corresponds to the differential of ${\sf Kos}_\bullet(\gg).$ The assertion follows. 
\end{proof}

Denote by $V^{\sf a}_\bullet$ the augmented Chevalley-Eilenberg resolution $V_\bullet,$ where  $V^{\sf a}_{-1}=\kk$ and $V^{\sf a}_0\to V^{\sf a}_{-1}$ is the augmentation $U\to \kk.$

\begin{lemma}\label{lemma_V_pure_acyclic} Under assumptions of Theorem \ref{theorem_CE_rel} the complex $V^{\sf a}_\bullet$ is pure acyclic. 
\end{lemma}
\begin{proof} Denote by $F_kV^{\sf a}_\bullet$ the augmented complexes $F_kV_\bullet,$ where  $F_kV_{-1}=\kk.$
First we prove that $F_kV_\bullet^{\sf a}$ is pure acyclic for each $k\geq 0.$ The proof is by induction. For $k=0$ the complex $F_0V^{\sf a}_\bullet$ is just the complex $0\to \kk\to \kk\to 0$ concentrated in degrees $0,-1.$ Hence, it is pure acyclic. Assume that $F_kV^{\sf a}_\bullet$ is pure acyclic and prove that $F_{k+1}V^{\sf a}_\bullet$ is pure acyclic. Lemma \ref{lemma_kos_GV_iso} implies that there is a short exact sequence
$$0 \longrightarrow F_kV^{\sf a}_\bullet \longrightarrow F_{k+1}V^{\sf a}_\bullet \longrightarrow {\sf Kos}_\bullet(\gg,k+1) \longrightarrow 0.$$
Theorem \ref{th_kos} implies that ${\sf Kos}_\bullet(\gg,k+1)$ is pure acyclic. By Lemma \ref{lemma_F_kV_pure} the subgroup $F_kV_n$ is pure in $V_n,$ and hence, it is pure in $F_{k+1}V_n.$ Then Lemma \ref{lemma_3x3} implies that $F_{k+1}V_\bullet'$ is pure acyclic. 

So, we proved that $F_kV^{\sf a}_\bullet$ is pure acyclic for any $k\geq 0.$ Since $V^{\sf a}_\bullet$ is a direct limit of $F_kV^{\sf a}_\bullet$ and the class of pure acyclic complexes is closed with respect to direct limits, $V_\bullet^{\sf a}$ is pure acyclic. 
\end{proof}

\begin{proof}[Proof of Theorem \ref{theorem_CE_rel}] The modules $U\otimes \Lambda^n \gg$ are $(U,\kk)$-projective. Then Lemma \ref{lemma_V_pure_acyclic} implies that $V_\bullet$ is a $\kk$-pure $(U,\kk)$-projective resolution. Then assertion follows from Theorem \ref{theorem_pure_rel_gen} and the obvious isomorphism $M\otimes_U V_\bullet\cong {\sf CE}_\bullet(\gg,M)$.
\end{proof}

\section{Relative simplicial homology}

Relative simplicial homology of a Lie algebra $\gg$ over a commutative ring $\kk$ is defined as comonad derived functors of the functor of abelianization ${\sf ab}: {\sf Lie}\to {\sf Mod}$ with respect to the comonad $\mathcal G^{\sf M}$ corresponding to the free forgetful adjunction with the category of modules $ \FF^{\sf M}:{\sf Mod} \leftrightarrows {\sf Lie}: \UU^{\sf M}.$
$$H^{\sf RS}_{*+1}=L_*^{\GG^{\sf M}}{\sf ab}$$
The aim of this section is to prove that over principal ideal domains relative simplicial homology is isomorphic to  the homology of Chevalley-Eilenberg complexes (and hence with relative Tor homology by Theorem \ref{theorem_CE_rel}), and show that for general rings it is not true. In this section we use the following simplified notation 
$$\FF(M):=\FF^{\sf M}(M).$$

\begin{lemma}\label{lemma_RT-hom_of_free_rel}
For any commutative ring $\kk$ and a module $M$ over $\kk$ the higher relative Tor and simplicial homology of $\FF(M)$ vanishes:
$$H^{\sf RT}_n(\FF(M))=0, \hspace{1cm} H^{\sf RS}_n(\FF(M))=0,  \hspace{1cm} n\geq 2.$$
\end{lemma}
\begin{proof}
Since the composition of adjunctions is an adjunction, we have that the universal enveloping algebra of $\FF(M)$ is the tensor algebra 
$U(\FF(M)) \cong  T(M).$ 
The augmentation ideal of $T(M)$ is $\bigoplus_{i\geq 1} M^{\otimes i},$ and hence, $I(T(M))\cong M\otimes T(M).$ It follows that $I(T(M))$ is $(T(M),\kk)$-projective. Therefore $$\dots \to 0\to I(T(M))\to T(M)$$ is a $\kk$-split $(T(M),\kk)$-projective resolution of $\kk.$ Then $H^{\sf RT}_n(\FF(M))=0.$

Comonad derived functors with respect to the comonad $\GG^{\sf M}$ coincide with Tierney-Vogel derived functors, with respect to the projective class of Lie algebras $\gg$ such that the counit $\varepsilon: \GG^{\sf M}(\gg)\to \gg$ splits \cite[Th. 3.1]{tierney1969simplicial}. Since $\FF(M)$ is in the projective class, we obtain
$L_n^{\GG^{\sf M}}{\sf ab}(\FF(M))=0$ for $n\geq 1.$  
\end{proof}

\begin{theorem}\label{th_RS} Under assumption of Theorem \ref{theorem_CE_rel} there are natural isomorphisms
$$H^{\sf RS}_*(\gg)\cong H^{\sf CE}_*(\gg)\cong H^{\sf RT}_*(\gg).$$
\end{theorem}
\begin{proof}
By Theorem \ref{theorem_CE_rel} we know that $H^{\sf CE}_*(\gg)\cong H^{\sf RT}_*(\gg).$ So we only need to prove  $H^{\sf RS}_*(\gg)\cong H^{\sf CE}_*(\gg).$

(1) Assume that $\kk$ is a principal ideal domain. In order to prove the isomorphism $H^{\sf RS}_*(\gg)\cong H^{\sf CE}_*(\gg)$ we prove that the shifted version of the Chevalley-Eilenberg complex  ${\sf CE}_\bullet'(\gg)$ is a computing complex for $L^{\GG^{\sf M}}_*{\sf ab}:{\sf Lie}\to {\sf Mod}.$ (C1) is obvious. By Theorem \ref{theorem_CE_rel} and Lemma \ref{lemma_RT-hom_of_free_rel} we obtain $$H_n({\sf CE'}_\bullet(\FF(M)))= H^{\sf CE}_{n+1}(\FF(M))=H^{\sf RT}_{n+1}(\FF(M))=0 $$ for $n\geq 1$ and any $\kk$-module $M.$ This implies (C2). (C3) follows from the fact that ${\sf CE}'_n(\gg)=\Lambda^{n+1}\gg$ depends only on the underlying structure of $\kk$-module on $\gg.$

(2) Assume that $\gg$ is flat over $\kk$. Consider the categories ${\sf FlatLie}$ and ${\sf FlatMod}$ of flat Lie algebras and flat modules. Since $\FF:{\sf Mod}\to {\sf Lie}$ commutes with filtered colimits, sends free modules to Lie algebras which are free as modules, and that a flat module is a filtered colimit of free modules, we obtain that $\FF$ sends flat modules to flat Lie algebras. It follows that the free-forgetful adjunction $\FF : {\sf Mod} \rightleftarrows {\sf Lie}:\UU$ can be restricted to an adjunction $\bar \FF : {\sf FlatMod} \rightleftarrows {\sf FlatLie}: \bar \UU.$ Hence the comonad $\GG^{\sf M}$ on ${\sf Lie}$ can be restricted to a comonad $\bar \GG^{\sf M} $ on ${\sf FlatLie}$ such that $L^{\GG^{\sf M}}_*{\sf ab}(\gg)=L^{\bar \GG^{\sf M}}_*{\sf ab}(\gg)$ for any flat Lie algebra $\gg$. Similarly to the case (1) one can prove that ${\sf CE}_\bullet'$ is a computing complex for $L^{\bar \GG^{\sf M}}_*{\sf ab}:{\sf FlatLie}\to {\sf FlatMod}.$ The assertion follows.  
 \end{proof}
 
 We say that a Lie algebra $\gg$ over $\kk$ is flat, if it is flat as a module over $\kk.$ The aim of this section is to prove the following theorem. 

\begin{theorem}\label{th_flat}
Let $\kk$ be a commutative ring and $\gg$ be a flat Lie algebra over $\kk.$ Then all versions of homology that we consider are naturally isomorphic
$$H^{\sf S}_*(\gg) \cong H^{\sf T}_*(\gg)\cong H^{\sf CE}_*(\gg)\cong  H^{\sf RT}_*(\gg)\cong H^{\sf RS}_*(\gg).$$
\end{theorem}
\begin{proof}
Theorem \ref{th_RS} implies that $H^{\sf CE}_*(\gg)\cong  H^{\sf RT}_*(\gg)\cong H^{\sf RS}_*(\gg).$ Lemma \ref{lemma_flatU} implies that $U\gg$ is flat.  Then by Proposition \ref{prop_tor_rel-ord} we obtain $H_*^{\sf RT}(\gg)\cong H_*^{\sf T}(\gg).$ So we only need to prove $H^{\sf S}_*(\gg)\cong H_*^{\sf CE}(\gg).$ By Theorem \ref{th_CE} it is enough to prove that $L_n\Lambda^m(\gg)=0$ for $n\ne 0.$ It follows from the fact that the underlying module of $\gg$ is a filtered colimit of free modules and the fact that $L_n\Lambda^m$ commutes with filtered colimits (Lemma \ref{lemma_commuting_filt_colim}). 
\end{proof}

\subsection{Example.} 

Further we construct an example that shows that the Theorem \ref{th_RS} does not hold for arbitrary commutative rings and arbitrary Lie algebras. 

\begin{proposition}\label{prop_counter_RSCE}
Let $\kk=\mathbb F_3[t_1,t_2,t_3,t_4]/(t_1^2,t_2^2,t_3^2,t_4^2),$ $M$ be a $\kk$-module generated by elements $e_1,e_2,e_3,e_4$ modulo relation $t_1e_1+t_2e_2+t_3e_4+t_4e_4=0$ $$M=\kk^4/\langle  t_1e_1+t_2e_2+t_3e_4+t_4e_4 \rangle$$
(as in Proposition \ref{prop_counter_Kos} for $n=p=3$) and $\gg=\FF(M)$ . Then 
$$H_3^{\sf CE}(\gg)\ne 0=H^{\sf RS}_3( \gg).$$
\end{proposition}
\begin{proof}  By Lemma \ref{lemma_RT-hom_of_free_rel} we have $H^{\sf RS}_3(\FF(M))=0$. Prove that $H_3^{\sf CE}(\FF(M))\ne 0.$  Note that there is a natural grading $\FF(M)=\bigoplus_{n\geq 1} \FF(M)_n,$ where $\FF(M)_n$ is generated as a $\kk$-module by commutators of weight $n$ (this grading comes from the obvious grading of the free non-associative algebra generated by $M$). It is easy to see that $\FF(M)_1\cong M$ and $\FF(M)_2\cong \Lambda^2 M, [m_1,m_2] \mapsto m_1\wedge m_2.$ The grading on the Lie algebra induces a natural grading  on ${\sf CE}_\bullet(\FF(M))$ such that its third component is the following complex
$$\dots \longrightarrow 0 \longrightarrow \Lambda^3M \overset{\partial_3}\longrightarrow M\otimes  \FF(M)_2 \longrightarrow \FF(M)_3,$$
where 
$$\partial_3(a_1\wedge a_2 \wedge a_3 )= - a_3\wedge  [a_1,a_2]+ a_2 \wedge [a_1,a_3] - a_1 \wedge  [a_2,a_3] $$
(see \cite[Prop. 2.2]{ivanov2020simple}). If we compose it with the isomorphism $\FF(M)_2\cong \Lambda^2M,$ we obtain that this map coincides with the map $\Lambda^3 M\to M\otimes \Lambda^2M$ from the Koszul complex ${\sf Kos}_\bullet(M,3).$ Then by Proposition \ref{prop_counter_Kos} we have ${\sf Ker}(\partial_3)\ne 0.$ It follows that $H^{\sf CE}_3(\FF(M))\ne 0.$ 
\end{proof}

\section{The second homology and Hopf's formula}

The first part of the following lemma is well known (see \cite[Exercise VI.19]{cartan1999homological}) but we add it for the sake of completeness. 

\begin{lemma}\label{lemma:tor(r/a,r/b)} Let $R$ be a ring,  $\mathfrak{a}$ be its right ideal  and $\mathfrak{b}$ be its left ideal. Then 
\begin{enumerate}
\item there are isomorphisms 
\begin{subequations}
\begin{align}
{\sf Tor}_1^R(R/\aa, R/\bb) &\cong \frac{\aa \cap \bb }{\aa\bb}, \\
{\sf Tor}_2^R(R/\mathfrak{a}, R/\mathfrak{b}) &\cong {\sf Ker}(\mathfrak{a}\otimes_R\mathfrak{b}\longrightarrow \mathfrak{a}\mathfrak{b}),\\
{\sf Tor}^R_n(R/\aa,R/\bb) &\cong {\sf Tor}^R_{n-2}(\aa,\bb), \hspace{1cm} n\geq 3,
\end{align}
\end{subequations}
where $\mathfrak{a}\otimes_R\mathfrak{b}\to \mathfrak{a}\mathfrak{b}$ is the multiplication map;
\item if $S$ is a subring of $R$ and the short exact sequences $$ 0\to \mathfrak{a} \to R \to R/\mathfrak{a}\to 0, \hspace{1cm} 0 \to \mathfrak{b} \to R \to R/\mathfrak{b}\to 0$$
are $S$-split, then there are  isomorphisms 
\begin{subequations}
\begin{align}
{\sf Tor}_1^{(R,S)}(R/\aa, R/\bb) &\cong {\sf Tor}_1^{R}(R/\aa, R/\bb),  \\
{\sf Tor}_2^{(R,S)}(R/\mathfrak{a}, R/\mathfrak{b}) &\cong {\sf Tor}_2^R(R/\mathfrak{a}, R/\mathfrak{b}), \label{eq:tor_2}\\ 
{\sf Tor}^{(R,S)}_n(R/\aa,R/\bb) &\cong {\sf Tor}^{(R,S)}_{n-2}(\aa,\bb), \hspace{1cm} n\geq 3.
\end{align}
\end{subequations}

\end{enumerate}
\end{lemma}

\begin{proof}
(1) Considering the long exact sequence for ${\sf Tor}^R_*(-,R/\mathfrak{b})$ associated with the short exact sequence $\aa \to R\to R/\aa,$ we obtain 
\begin{equation*}
\begin{split}
{\sf Tor}^R_1(R/\aa,R/\bb)&\cong {\sf Ker}( \aa\otimes_R R/\bb \to R\otimes_R R/\bb),\\
{\sf Tor}^R_n(R/\aa, R/\bb)&\cong {\sf Tor}^R_{n-1}(\aa,R/\bb), \hspace{1cm}n\geq 2.
\end{split}
\end{equation*}

Using that the map $ \aa \otimes_R R/\bb \to R\otimes_R R/\bb $ is isomorphic to the map $\aa/(\aa\bb)\to R/\bb,$ we obtain the isomorphism ${\sf Tor}_1^R(R/\aa,R/\bb)\cong (\aa\cap \bb)/\aa\bb.$ 
The long exact sequence for ${\sf Tor}^R_*(\aa,-)$ associated with the short exact sequence $\bb \to R \to R/\bb $ implies 
\begin{equation*}
\begin{split}
{\sf Tor}_1^R(\aa,R/\bb)&\cong {\sf Ker}(\aa\otimes_R\bb \longrightarrow \aa\bb),\\
{\sf Tor}_n^R(\aa,R/\bb)&\cong{\sf Tor}^R_{n-1}(\aa,\bb), \hspace{1cm} n\geq 2. 
\end{split}
\end{equation*}
The assertion follows. 

(2) The proof of the second part is very similar to the proof of the first part. We just need to replace ${\sf Tor}^R_*$ by ${\sf Tor}^{(R,S)}_*$ and use the long exact sequence associated with $S$-split short exact sequences 
\cite[Ch.IX,(8.6)]{maclane2012homology}.
\end{proof}

\begin{proposition}\label{prop:T-RT}
Let $\kk$ be a commutative ring and $\gg$ be a Lie algebra over $\kk.$ Then there is a natural isomorphism
$$H_2^{\sf T}(\gg)\cong H_2^{\sf RT}(\gg).$$
\end{proposition}
\begin{proof}
It follows from the isomorphism \eqref{eq:tor_2} if we take $R=U\gg, S=\kk$ and $\aa=\bb=I,$ where $I$ is the augmentation ideal of $U\gg.$
\end{proof}

\begin{proposition}\label{prop:hopfCE}
Let $\kk$ be a commutative ring and $\gg\cong \ff/\rr$ be a Lie algebra over $\kk$ presented as a quotient of a free Lie algebra. Then 
$$H_2^{\sf CE}(\gg)\cong \frac{\rr \cap [\ff,\ff]}{[\rr,\ff]}.$$
\end{proposition}
\begin{proof}
We set $\tilde \Lambda^2 \gg  :={\sf Coker}(\Lambda^3 \gg \to \Lambda^2 \gg ),$ where $\Lambda^3 \gg \to \Lambda^2 \gg $ is the differential from ${\sf CE}_\bullet(\gg)$. Then 
$$ H_2^{\sf CE}(\gg) = {\sf Ker}( \tilde \Lambda^2 \gg \xrightarrow{\ [\cdot,\cdot]\ } \gg).$$
Since $H_2^{\sf CE}(\ff)=0,$ we obtain an isomorphism $\tilde \Lambda^2\ff \cong [\ff,\ff]$ given by $a\wedge b\mapsto [a,b].$ 

It is well known that for a short exact sequence $0\to A\to B\to C\to 0$ of $\kk$-modules there is an exact sequence $A\otimes B \to \Lambda^2 B \to \Lambda^2 C\to 0.$ 
Therefore we have the  exact sequence 
$\rr\otimes \ff\to \Lambda^2 \ff \to \Lambda^2 \gg \to 0.$ Taking the quotients by the images of the exterior cubes, we obtain an exact sequence  
$$\rr \otimes \ff \to \tilde \Lambda^2 \ff \to \tilde\Lambda^2 \gg\to 0.$$ If we combine it with the isomorphism $[\ff,\ff]\cong \tilde \Lambda^2 \ff,$ we obtain 
$$\tilde \Lambda^2 \gg \cong {\rm Coker}(\rr\otimes\ff \to [\ff,\ff])\cong \frac{[\ff,\ff]}{[\rr,\ff]}.$$ 
Therefore $H_2^{\sf CE}(\gg)={\rm Ker}\left( \frac{[\ff,\ff]}{[\rr,\ff]} \to \gg\right) = \frac{\rr\cap [\ff,\ff]}{[\rr,\ff]}.$
\end{proof}

\begin{theorem} \label{th_hopf}
Let $\kk$ be a principal ideal domain and $\gg=\ff/\rr$ be a Lie algebra over $\kk$ presented as a quotient of a free algebra. Then all types of the second homology that we consider are naturally isomorphic and isomorphic to $\frac{\rr \cap [\ff,\ff]}{[\rr,\ff]}:$
$$H_{2}^{\sf CE}(\gg)\cong H_2^{\sf S}(\gg)\cong H_2^{\sf RS}(\gg)\cong H^{\sf RT}_2(\gg) \cong  H_2^{\sf T}(\gg)\cong \frac{\rr \cap [\ff,\ff]}{[\rr,\ff]}.$$
\end{theorem}
\begin{proof} By Corollary \ref{cor_H_2S} we have $H_2^{\sf CE}(\gg)\cong H_2^{\sf S}(\gg).$ 
By Theorem \ref{th_RS} over a principal ideal domain we have  $H_2^{\sf CE}(\gg)\cong H_2^{\sf RS}(\gg)\cong  H_2^{\sf RT}(\gg).$  By Proposition \ref{prop:T-RT} we have an isomorphism $H_2^{\sf RT}(\gg) \cong H_2^{\sf T}(\gg).$ By Proposition \ref{prop:hopfCE} we have $H_2^{\sf CE}(\gg)\cong \frac{\rr\cap [\ff,\ff]}{[\rr,\ff]}.$
\end{proof}

The rest of this section is devoted to constructing a commutative ring $\kk$ and Lie algebra $\gg$ over $\kk$ such that $H_2^{\sf S}(\gg)\not\cong H_2^{\sf RT}(\gg).$ 

Note that for any $\kk$-module $M$ the universal enveloping algebra of $\FF(M)$ is the tensor algebra
$$U(\FF(M))\cong T(M).$$ In particular, we have a morphism of Lie algebras 
$\FF(M)\to T(M)$ which is not necessarily injective because the Poincar\'e–Birkhoff–Witt property holds only under certain conditions. We set 
$$\FF'(M):={\rm Im}(\FF(M)\to T(M)).$$
\begin{lemma}\label{lemma_F'}
For any commutative ring $\kk$ and any $\kk$-module $M$ 
$$H_n^{\sf RT}(\FF'(M))=0, \hspace{1cm} n\geq 2.$$
\end{lemma}
\begin{proof}
Since the map $\FF(M)\to \FF'(M)$ is surjective, the map $T(M)=U\FF(M)\to U\FF'(M)$ is also surjective. On the other hand the map $\FF(M)\to T(M)$ factors as $\FF(M)\to \FF'(M)\to T(M).$ Therefore, the identity map $T(M)\to T(M)$ factors as $T(M)\to U\FF'(M)\to T(M),$ where $T(M)\to U\FF'(M)$ is an epimorphism. It follows that $U\FF'(M)\cong T(M).$ Hence $H^{\sf RT}_*(\FF'(M))=H^{\sf RT}(\FF(M)).$ The assertion follows from Lemma \ref{lemma_RT-hom_of_free_rel}.
\end{proof}

\begin{proposition}\label{prop_counter_H_2}
Let $\kk=\mathbb F_2[t_1,t_2,t_3]/(t_1^2,t_2^2,t_3^2)$ and $M=\kk^3/\langle t_1e_1+t_2e_2+t_3e_3\rangle$ (as in Proposition \ref{prop_counter_Kos} for $n=p=2$). Assume that $\gg=\FF'(M).$ Then 
$$H^{\sf CE}_2(\gg)\ne 0=H^{\sf RT}_2(\gg).$$
\end{proposition}
\begin{proof} The equation $H^{\sf RT}_2(\gg)=0$ follows from Lemma \ref{lemma_F'}. The algebra $\gg$ has a natural grading $\gg=\bigoplus_{n\geq 1} \gg_n$ by the weight of commutators such that $\gg_n=\gg\cap M^{\otimes n}.$ Then the complex ${\sf CE}_\bullet(\gg)$ inherits a natural grading, whose second component is the complex $0\to \Lambda^2\gg_1 \to \gg_2\to 0.$ Then it is enough to prove that the kernel of the map $\Lambda^2\gg_1 \to \gg_2$ is nontrivial. 
Since $\gg_1=M$ and $\gg_2\subseteq M\otimes M,$ the kernel of this map is equal to the kernel of the map $\partial:\Lambda^2 M \to M\otimes M$ which is nontrivial by Proposition \ref{prop_counter_Kos}. 
\end{proof}

\section{Appendix: PBW-property}

An example of a Lie algebra over a commutative ring that does not satisfy  the Poincar\'e-Birkhoff-Witt property is given in \cite[\S 5]{cohn1963remark}. However
some details of the proof are omited. Here we present a related example with this property with a more detailed proof which is based on our counterexample to acyclicity of the Koszul complex (Corollary  \ref{cor_counter_kos}). 

For any module $M$ over any commutative ring $\kk$ one can consider the module 
$$\gg:=M\oplus \Lambda^2 M$$ 
as a Lie algebra over $\kk$ taking $[a,b]=a\wedge b$ for $a,b\in M$ and $[-,\alpha]=[\alpha,-]=0$ for   $\alpha\in \Lambda^2M.$ This is a nilpotent Lie algebra of class 2.  

\begin{proposition}
 Let $M$ be a $\kk$-module such that the map $$\partial: \Lambda^2 M\to M\otimes M, \hspace{1cm} \partial(a\wedge b)= a\otimes b - b\otimes a$$ is not injective (as in Corollary \ref{cor_counter_kos}) and $\gg=M\oplus \Lambda^2 M$. Then the map to the universal enveloping algebra $$\varphi: \gg\to U\gg$$ is not injective. 
\end{proposition}
\begin{proof}  Since $\gg=M\oplus \Lambda^2 M$ is a direct sum, $\varphi$ consists of two components $\varphi_1:M\to U\gg$ and $\varphi_2:\Lambda^2 M \to U\gg.$ It is easy to check that 
$$\varphi_2=\mu \circ (\varphi_1\otimes \varphi_1)\circ \partial,$$
where $\mu:U\gg \otimes U\gg \to U\gg$ is the  multiplication map. 
Hence $\varphi_2$ is not injective.
\end{proof}

\printbibliography

\end{document}